\documentclass[11pt]{amsart}
\usepackage{amsfonts}
\usepackage{amsfonts,latexsym,rawfonts,amsmath,amssymb,amsthm, a4,a4wide}

\usepackage[plainpages=false]{hyperref}

\usepackage{graphicx}

\RequirePackage{color}
\usepackage[margin=0.99in]{geometry}

\numberwithin{equation}{section}

\newcommand{\beq}{\begin{equation}}
\newcommand{\eeq}{\end{equation}}
\newcommand{\beqs}{\begin{eqnarray*}}
\newcommand{\eeqs}{\end{eqnarray*}}
\newcommand{\beqn}{\begin{eqnarray}}
\newcommand{\eeqn}{\end{eqnarray}}
\newcommand{\beqa}{\begin{array}}
\newcommand{\eeqa}{\end{array}}
\newcommand{\R}{\mathbb R}
\newcommand{\calF}{{\mathcal F}}

\newcommand{\e}{\varepsilon}
\newcommand{\p}{\partial}

\newenvironment{myindentpar}[1]%
{\begin{list}{}%
         {\setlength{\leftmargin}{#1}}%
         \item[]%
}
{\end{list}}

\newtheorem{prop}{Proposition}[section]
\newtheorem{thm}[prop]{Theorem}
\newtheorem{lem}[prop]{Lemma}

\newtheorem{rem}[prop]{Remark}

\renewcommand{\div}{\mbox{div}\,}
\newcommand{\trace}{\mbox{trace}\,}
\newcommand{\dist}{\text{dist}}

\author{Nam Q. Le}
\address{Department of Mathematics, Indiana University, 
Bloomington, IN 47405, USA. }
\email {nqle@indiana.edu}
\author{Bin Zhou}
\address{School of Mathematical Sciences, Peking University, Beijing, 100871, China}
\email{bzhou@pku.edu.cn}
\thanks{The research of the first author was supported in part by NSF grant DMS-1764248.  The second author was supported by NSFC 11571018 and 11822101}
\allowdisplaybreaks
\title[Singular fourth order equations of Monge-Amp\`ere type]{Solvability of a class of singular fourth order equations of Monge-Amp\`ere type}

\makeatletter
\@namedef{subjclassname@2020}{%
  \textup{2020} Mathematics Subject Classification}
\makeatother

\begin{document}
\begin{abstract}
We study the solvability of the second boundary value problem for a class of highly singular fourth order equations of Monge-Amp\`ere type. They arise in the approximation of convex functionals subject to a convexity constraint using Abreu type equations. 
Both the Legendre transform and partial Legendre transform are used in our analysis. In two dimensions, we establish global solutions to the second boundary value problem for highly singular Abreu equations where the right hand sides are of $q$-Laplacian type for all $q>1$. We show that minimizers of variational problems with a convexity constraint in two dimensions that arise from the Rochet-Chon\'e model in the monopolist's problem in economics with $q$-power cost can be approximated in the uniform norm by solutions of the Abreu equation for a full range of $q$.
\end{abstract}

\maketitle

\section{Introduction}
\subsection{Singular fourth order equations of Monge-Amp\`ere type} This paper is concerned with a class of 
fourth order equations of 
Monge-Amp\`ere type 
\beq\label{4-eq}
\sum_{i, j=1}^n U^{ij}D_{ij}[G'(\det D^2 u)]\equiv \trace\left[(\det D^2 u) (D^2 u)^{-1} D^2 (G'(\det D^2 u))\right]=f \quad\text{in }\Omega\subset \R^n.
\eeq
Here and throughout, $n\geq 2$, $U=(U^{ij})_{1\leq i, j\leq n}=(\det D^2 u)(D^2 u)^{-1}$ is the cofactor matrix of the Hessian matrix $D^2 u=(D_{ij}u)_{1\leq i, j\leq n}\equiv \left(\frac{\p^2 u}{\p x_i \p x_j}\right)_{1\leq i, j\leq n}$
of an unknown uniformly convex function $u$, and $G:(0, \infty)\rightarrow \R$ is a smooth function satisfying certain conditions.
Equations of this type appear in many contexts ranging from affine geometry, complex geometry and economics.

Equation (\ref{4-eq}) with  the right hand side $f=f(x)$ has been extensively studied. When $G(t)=t^{\frac{1}{n+2}}$,  it is the {\it prescribed affine mean curvature equation} in affine geometry \cite{Ch}.
When $G(t) =\log t$, it is the {\it Abreu's equation} arising from the problem of finding extremal metrics on toric manifolds in K\"ahler geometry \cite{Ab}, and it is equivalent to
$$\sum_{i,j=1}^n\frac{\partial^2 u^{ij}}{\partial x_i\partial x_j}=f(x),$$
where $(u^{ij})$ is the inverse matrix of $D^2u$.
The regularity  and solvability of these equations and related geometric problems have been extensively investigated in the past two decades, including \cite{D1, D2, D3, D4, TW1, TW2, TW3, Z1, Z2, CW, CHLS, Le1, Le2}, to name a few. In all these works, the regularity theory  of linearized Monge-Amp\`ere equation, initiated in the fundamental work of Caffarelli-Guti\'errez \cite{CG}, plays an important role.  The second order operator $L_u:=U^{ij} D_{ij}$ is called a linearized Monge-Amp\`ere operator because
the coefficient matrix $(U^{ij})$ comes from the linearization of the Monge-Amp\`ere operator: 
$$U=\frac{\p (\det D^2 u)}{\p (D^2 u)}.$$
One can also note that $L_u v= U^{ij} D_{ij} v= \text{trace} (UD^2 v)$ is the coefficient of $t$ in the expansion
$$\det D^2 (u + t v) =\det D^2 u + t~ \text{trace} (U D^2 v) + \cdots+ t^n \det D^2 v.$$

Recently, a class of fourth order equations of 
Monge-Amp\`ere type in the general form of 
\beq\label{4-eq-g}
U^{ij}D_{ij}[G'(\det D^2 u)]=f(x, u, Du, D^2u) 
\eeq 
has been introduced by the first author in \cite{LeCPAM, Le6, Le7} in the study of convex functionals with a convexity constraint related to the Rochet-Chon\'e model \cite{RC} for the monopolist's problem in economics; see also \cite{CR} for the precursor of this class of equation. One usually takes $G(t) = \log t$ (the Abreu type equation) or $G(t) = t^{\theta}$ where $0<\theta<1/n$ (the affine maximal surface type equation). If we view the linearized Monge-Amp\`ere operator as a non-divergence form operator, then for regularity theory, a natural integrability condition for the right hand side is $L^n$. In general, the regularity theory of linearized Monge-Amp\`ere equation with $L^n$ right hand side in previous works, including those in \cite{CG, GN1, GN2}, does not apply to \eqref{4-eq-g}, since the right hand side of (\ref{4-eq-g}) depends on $D^2u$ which is a priori only a matrix-valued measure. This is the reason why (\ref{4-eq-g}) was called {\it singular Abreu equation} in the above mentioned works. As mentioned in \cite{LeCPAM}, even when $f(x, u, Du, D^2u) = \Delta u$, the right hand side of (\ref{4-eq-g}) has low integrability, which is at most $L^{1+\e}$ for some small constant $\e$.  However, the regularity results for the linearized Monge-Amp\`ere equation with $L^{n/2+\e}$ right hand side in \cite{LN}, which makes use of the divergence form character of $L_u$ (note that, $\sum_{i=1}^n D_i U^{ij}=0$ for all $j$) via the optimal integrability of its Green's function, allow the first author to study (\ref{4-eq-g}) in two dimensions under suitable assumptions on $f$ and boundary data. In light of this, it is natural to consider (\ref{4-eq-g}) when $f$ has a divergence form structure, of which $f= \Delta u= \div (Du)$ is a particular example.

When $f$ is just the divergence of a vector field, such as $\div (|Du|^{q-2} Du)$ where $1<q<2$, the right hand side of (\ref{4-eq-g}), which contains the term $|Du|^{q-2} \Delta u$,  become more singular in $D^2 u$ and hence it does not belong to  $L^{1+\e}$. Equation (\ref{4-eq-g}) with this type of {\it highly singular} right hand side $f$ arises from the analysis of the Rochet-Chon\'e model \cite{RC} with $q$-power cost; see Section \ref{qcostsect}. Despite this singularity, $\div (|Du|^{q-2} Du)$ is the divergence of a bounded vector field when $Du$ is bounded.
In dimension two, based on the $L^{1+\e}$ integrability of the Green's function of the linearized Monge-Amp\`ere operator, interior and global H\"older regularity estimates for linearized Monge-Amp\`ere equation with the right hand being the divergence of a bounded vector field have been established in \cite{Le3, Le4}. The higher dimensional case is widely open. Nevertheless, 
these results suggest the possibility of solving certain boundary value problems for (\ref{4-eq-g}) in the case of highly singular right hand side. To apply the results in \cite{Le3, Le4}, a key step is to obtain the positive lower and upper bounds for the Hessian determinant $\det D^2 u$ of a solution $u$. 
However, the  techniques in \cite{LeCPAM, Le6, Le7} could not handle the highly singular Abreu type equations.

In this paper, we use the Legendre transform and partial Legendre transform to investigate the regularity and solvability of highly singular equation \eqref{4-eq-g}
which arises from a variational problem. This is the case when $f$ has the form
\beq\label{cond1}
f(x, u, Du, D^2u)=F_z-\sum_{i=1}^n F_{p_ix_i}-\sum_{i=1}^n F_{p_i z}D_i u-\sum_{i,j=1}^n F_{p_i p_j}D_{ij}u
\eeq
where $F=F(x, z, p)$ is a function on $\mathbb R^n\times \mathbb R \times \mathbb R^{n}$. In order words, (\ref{4-eq-g})
is the Euler-Lagrange equation of the functional
\beq\label{func0}
\mathcal F(u)=\int_\Omega F(x, u, Du)\, dx-\int_\Omega G(\det D^2u)\,dx.
\eeq
The Legendre transform has been extensively used in the analysis of the Monge-Amp\`ere type equations including the fourth order (\ref{4-eq}) with $f$ depending only on the independent variable $x$; see, for example \cite{CHLS, Le1, TW1, TW2, Z1, Z2} and the references therein. Likewise, the partial Legendre transform has been widely used to investigate fine properties of Monge-Amp\`ere equations, especially degenerate ones; see, for example \cite{DS, GP, LS, Li} 
and the references therein. However, to the best of our knowledge, it has not been used in fourth order equations of Monge-Amp\`ere type before. For equation of the type (\ref{4-eq-g}), the Legendre transform usually gives one-sided bound for $\det D^2 u$. In two dimensions, it turns out out the other one-sided bound can be obtained using the partial Legendre transform.

 Note that (\ref{4-eq-g}) can be written as a system of two equations for $u$ and $w=G'(\det D^2 u)$. One is a Monge-Amp\`ere equation for $u$ in the form of $$\det D^2 u=
(G')^{-1}(w)$$ and other is a linearized Monge-Amp\`ere equation for $w$ in the form of $$L_u w=U^{ij} D_{ij }w=f(\cdot, u, Du, D^2 u).$$ Thus, a very natural boundary value problem for (\ref{4-eq-g}) is the second boundary value problem where one describes the values of $u$ and $w$ on the boundary $\p\Omega$.   In what follows, when $u$ is only $C^3$, the expression $\sum_{i, j=1}^{n}U^{ij}D_{ij}w$ is understood as $\sum_{i, j=1}^{n}D_i(U^{ij}D_{j}w)$.
\subsection{Solvability of the second boundary value problem for highly singular Abreu equations}
Let $q\in (1,\infty)$. Let $\Omega$ be an open, smooth, bounded and uniformly convex domain in $\R^n$. Let $\varphi\in C^{5}(\overline{\Omega})$, $\psi\in C^{3}(\overline{\Omega})$ with $\inf_{\p\Omega}\psi>0$. 

We are interested in the second boundary value problem of the Abreu equation with right hand side of $q$-Laplacian type for a uniformly convex function $u$:
\begin{equation}
\label{Abreu}
  \left\{ 
  \begin{alignedat}{2}\sum_{i, j=1}^{n}U^{ij}D_{ij}w~& =-\div (|Du|^{q-2} Du)+ F^0_z(x, u)~&&\text{in} ~\Omega, \\\
 w~&= (\det D^2 u)^{-1}~&&\text{in}~ \Omega,\\\
u ~&=\varphi~&&\text{on}~\p \Omega,\\\
w ~&= \psi~&&\text{on}~\p \Omega.
\end{alignedat}
\right.
\end{equation}
Here $F^0(x, z)$ is a function on $\mathbb R^n\times\mathbb R$.
The particular form of the right hand side of (\ref{Abreu}) was partially motivated by problems from economics; see (\ref{Abreuq}).
Of course, one can replace the term $-\div (|Du|^{q-2} Du)$ in (\ref{Abreu}) by other expressions involving $Du$ and $D^2u$. We chose this particular form in (\ref{Abreu}) due to its simplicity; moreover, this form still captures the highly singular feature of the right hand side in the Hessian $D^2u$ which is the case  when $1<q<2$ and $Du$ is small, or when $q>2$ and $Du$ is large. It should be emphasized that the negative sign in $-\div (|Du|^{q-2} Du)$ in (\ref{Abreu}) is crucial. If this term is replaced by $\div (|Du|^{q-2} Du)$, then (\ref{Abreu}) might not have a global, smooth and uniformly convex solution; see \cite[Remark 1.5]{Le6}.

  In two dimensions, equation (\ref{Abreu}) in the case of $q\geq 2$ and $F^0_z\equiv 0$ was considered in \cite[Theorem 2.6]{LeCPAM}. The case $1<q<2$, even if $F^0_z\equiv 0$, was left open.
  Also left open in the two dimensions was the case of $q\geq 2$ and $F^0_z\not\equiv 0$.

  The solvability of (\ref{Abreu}) is usually established using higher order a priori estimates and degree theory. As mentioned earlier, a critical step is to 
obtain the a priori lower and upper bounds for $\det D^2 u$ when $1<q<2$ and $F^0_z\not\equiv 0$. All known arguments in obtaining the lower bound for $\det D^2 u$ use two ingredients:
  (i) $ \div (|Du|^{q-2} Du) \leq C \Delta u$ when $|Du|$ is bounded; and (ii) $\trace (U^{ij})=\Delta u$.  Thus, they are applicable only to $q\geq 2$ and $n=2$. 
Here we use the Legendre transform to study (\ref{Abreu}).
 We resolve the remaining cases $1<q<\infty$ and $F^0_z\not\equiv 0$ in dimensions $n=2$.
 
 We assume that $F^0$ is smooth and satisfies
 \begin{equation}
 \label{Qcond}
 F^0_z(x, z) \leq \omega(|z|);\quad 
 -F^0_z(x, z)(z-\tilde z)\leq \omega (|\tilde z|)\quad \text{for all } x\in\Omega,\quad \text{and all }z, \tilde z\in \R
 \end{equation}
 where $\omega:[0,\infty)\rightarrow [0,\infty)$ is a continuous and increasing function.  
 
 Our first main theorem states as follows.
 \begin{thm}[Solvability of the second boundary value problem for highly singular Abreu equations]
\label{SBVq12}
Let $\Omega\subset\R^2$ be an open, smooth, bounded and uniformly convex domain.  Let $q>1$. Assume (\ref{Qcond}) holds.
Assume that $\varphi\in C^{5}(\overline{\Omega})$ and $\psi\in C^{3}(\overline{\Omega})$ with $\inf_{\p \Omega}\psi>0$.
Consider  the following
second boundary value problem:
\begin{equation}
\label{Abreu4}
  \left\{ 
  \begin{alignedat}{2}\sum_{i, j=1}^{2}U^{ij}D_{ij}w~& =-div (|Du|^{q-2} Du) + F^0_z(x, u)~&&\text{in} ~\Omega, \\\
 w~&= (\det D^2 u)^{-1}~&&\text{in}~ \Omega,\\\
u ~&=\varphi~&&\text{on}~\p \Omega,\\\
w ~&= \psi~&&\text{on}~\p \Omega.
\end{alignedat}
\right.
\end{equation}
\begin{myindentpar}{1cm}
(i) If $q\geq 2$, then there exists a  uniformly convex solution $u\in C^{4, \beta}(\overline{\Omega})$ to (\ref{Abreu4}) with 
$$\|u\|_{C^{4,\beta}(\overline{\Omega})}\leq C$$
for some $\beta\in (0, 1)$ and $C>0$ depending on $q,\Omega$, $\omega$, $F^0$, $\varphi$ and $\psi$. \\
(ii) If $1<q<2$, then there exists a  uniformly convex solution $u\in C^{3, \beta}(\overline{\Omega})$ to (\ref{Abreu4}) with 
$$\|u\|_{C^{3,\beta}(\overline{\Omega})}\leq C$$
for some $\beta\in (0, 1)$ and $C>0$ depending on $q,\Omega$, $\omega$, $F^0$, $\varphi$ and $\psi$.
\end{myindentpar}
\end{thm}

Theorem \ref{SBVq12} will be proved in Section \ref{q12esect}. The main idea of the proof of Theorem \ref{SBVq12} is to use partial Legendre transform. After the partial Legendre transformation, the first two equations of  \eqref{Abreu4} become a quasi-linear elliptic equation for the dual $w^\star$ of $w$.
To estimate  the a priori lower and upper bounds for $\det D^2 u$ when $1<q<2$ and $F^0_z\not\equiv 0$, we need the $C^2$ character of $w^\star$ (in order to apply the maximum principle to an elliptic equation in non-divergence form) which is equivalent to $u$ being $C^4$. This is not possible for $q\in (1,2)$. Thus, we will not apply the partial Legendre transform directly to \eqref{Abreu4}. Instead, we apply it to its approximation (\ref{Abreu4e}) whose global $C^{4}$ solutions are guaranteed.

The Legendre transform can also be used to establish interior higher order derivative estimates in higher dimensions for (\ref{Abreu}) when $F^0_z\leq 0$. This is the content of our next theorem.
\begin{thm}[Interior higher order derivative estimates for highly singular Abreu equations]
  \label{SBVn3}
  Let $n\geq 3$. Let $u \in C^{3}(\Omega)
  \cap C^2(\overline{\Omega})$ be a uniform convex solution to \eqref{Abreu} where  $F^0$ is smooth, $F^0_z\leq 0$, $\varphi\in C^{5}(\overline{\Omega})$, $\psi\in C^{3}(\overline{\Omega})$ with $\inf_{\p\Omega}\psi>0$. Then, for any $\Omega'\Subset\Omega$, we have
  $$\|u\|_{C^{4,\alpha}(\Omega')}\leq C\quad \text{if } q\geq 2$$
  and
    $$\|u\|_{C^{3,\alpha}(\Omega')}\leq C \quad \text{if } 1<q<2$$
  where $\alpha$ and $C$ depend on $\varphi, \psi,  F^0$, $\Omega, n, q$ and $\dist (\Omega',\p\Omega)$.
    \end{thm}
We will prove Theorem \ref{SBVn3} in Section \ref{Legsect}. 
However, due to the lack of global regularity in higher dimensions for the linearized Monge-Amp\`ere equation with right hand side being the divergence of a bounded vector field, it is still an open problem to solve the second boundary value problem for \eqref{Abreu} when $n\geq 3$. Note that, when $n\geq 3$ and the right hand side of (\ref{Abreu}) is replaced by $-\gamma \div (|Du|^{q-2} Du)$ where $q\geq 2$ and $\gamma>0$ is a small constant depending on $n, q, \varphi,\psi$ and $\Omega$, the existence of a unique global $C^{4,\beta}(\overline{\Omega})$ solution to (\ref{Abreu}) was established in \cite{Le6}.  It would be interesting to remove this smallness of $\gamma$.

\subsection{Approximations of minimizers of Rochet-Chon\'e model with non-quadratic costs}

\label{qcostsect}
Let $\Omega_0$, $\Omega$ be bounded, open, smooth, and convex domains in $\R^n$ where $\Omega$ contains $\overline{\Omega_0}$.
Let $\varphi\in C^{5}(\overline{\Omega})$ be a convex function.  Let $F(x, z, p):\R^n\times \R\times \R^n\rightarrow\R$ be the Lagrangian given by
$$F(x, z, p) = (|p|^q/q- x\cdot p)\gamma(x) + F^0(x, z)$$
where $\gamma$ is a nonnegative and Lipschitz function.
We assume the following convexity and growth assumptions on $F^0$:
\begin{equation}
\label{F0q}
\small
(F_z^0(x, z)-F_z^0(x, \tilde z))(z-\tilde z)\geq 0;\ \  |F_z^0(x, z)|+  |F^0(x, z)|\leq \eta (|z|)~\text{for all } x\in\Omega_0~\text{and } z,\tilde z\in\R
\end{equation}
where $\eta:[0,\infty)\rightarrow [0,\infty)$ is a continuous and increasing function.  

When $F^0(x, z) =z\gamma(x)$, the Lagrangian $F$ covers the Rochet-Chon\'e model with $q$-power cost and relative frequency of agents in the population given by $\gamma$; see \cite[p. 790]{RC}. 
We are interested in the following variational problem subject to a convexity constraint:
\begin{equation}
\label{pb1}
\inf_{u\in \bar{S}[\varphi,\Omega_0]} \int_{\Omega_0} F(x, u(x), Du(x)) \,dx
\end{equation}
where 
\begin{multline}
\label{barS1}
\bar{S}[\varphi, \Omega_0]=\{ u: \Omega_0\rightarrow \R\mid u \text{ is convex and admits a convex extension to\ } \Omega \text{\ such that }\\   u=\varphi\text{ on }\Omega\backslash \Omega_0\}.
\end{multline}
Since functions in $\bar{S}[\varphi, \Omega_0]$ are Lipschitz continuous with Lipschitz constants bounded from above by $\|D\varphi\|_{L^{\infty}(\Omega)}$, $\bar{S}[\varphi, \Omega_0]$ is compact in the topology of uniform convergence on compact subsets of subsets of $\Omega$. With (\ref{F0q}), one can show that (\ref{pb1}) has a minimizer in $\bar{S}[\varphi, \Omega_0]$.
 Heuristically, the boundary conditions for minimizers associated with (\ref{barS1}) are 
\begin{equation}
\label{FBV}
u=\varphi \quad\text{and }\frac{\p u}{\p \nu_0}\leq \frac{\p \varphi}{\p \nu_0}\quad \text{on }\p\Omega_0\quad \text{where }\nu_0 \text{ is the unit outer normal vector on } \p\Omega_0.
\end{equation}

In \cite{RC}, Rochet-Chon\'e modeled the monopolist problem in product line design with $q$-power  cost 
   using minimization, over convex functions $u\geq \varphi$, of the functional
\begin{equation}
\label{RCeg}
\Phi(u)=\int_{\Omega_0} \left[|Du(x)|^q/q  -x\cdot Du(x)  + u(x)\right] \gamma (x) ~dx.
\end{equation}

Here $ -\Phi(u)$ is  the monopolist's profit; 
$u$ is the buyers' indirect utility function with bilinear valuation;
  $\Omega_0\subset\R^n$ is the collection of  types of agents; $\gamma$ is the relative frequency of different types of agents in the population; the given convex function $\varphi$ is referred to as  the participation constraint.
The constraint (\ref{barS1}) can be heuristically viewed as a special case of the constraint $u\geq \varphi$ in $\Omega_0$.

The convexity constraints such as $u\geq \varphi$ in (\ref{RCeg}) and (\ref{barS1}) in (\ref{pb1}) pose serious challenges, as elucidated in \cite{BCMO, Mir},  in numerically computing minimizers of the above problems. 
This calls for robust approximation schemes for minimizers of variational problems with a convexity constraint.   
The question we would like to address here is how to approximate minimizers of (\ref{pb1})  
in the uniform norm by solutions of some higher order equations whose global well-posedness can be established. 
The approximating scheme proposed in \cite{LeCPAM, Le7} use the second boundary value problem of fourth order equations of Abreu type and it only works for $q=2$ and $n=2$; see also \cite{RC} for $F$ not depending on $p$. The reason $q=2$ is that the gradient-dependent term 
$F^1(x, p)= (|p|^q/q- x\cdot p)\gamma(x)$ of the Lagrangian $F$ was required to satisfy for some $C>0$
\begin{equation}
\label{F1}
0\leq F^1_{p_i p_j}(x, p)\leq C I_n;\ \  |F^1_{p_i x_i}(x, p)| \leq  C (|p| +1) \text{ for all }x\in\Omega_0~\text{and for each } i.
\end{equation}

Inspired by the approximation equation (\ref{Abreu4e}) in the proof of Theorem \ref{SBVq12}, we will answer positively the question of approximating minimizers of (\ref{pb1}) by solutions of the second boundary value problems of fourth order equations of Monge-Amp\`ere type for the full range $(1,\infty)$ of $q$.  The idea is to modify the schemes in \cite{LeCPAM, Le7} by further approximating the gradient-dependent term. We describe this scheme below.

Let $\rho$ be a uniformly convex defining function of $\Omega$, that is, 
\begin{equation}
\label{rhoeq}
\Omega:=\{x\in \R^n: \rho(x)<0\},~\rho=0 \text{ on } \p\Omega \text{ and }D\rho\neq 0 \text{ on }\p\Omega.
\end{equation} 
 For $\e>0$, let $\delta(\e)=\e$, and
 consider the following second boundary value problem for a uniform convex function $u_\e$:
\begin{equation}
\label{Abreuq}
  \left\{ 
  \begin{alignedat}{2}\e\sum_{i, j=1}^{n}U_\e^{ij}D_{ij}w_\e~& =f_{\e}~&&\text{in} ~\Omega, \\\
 w_\e~&= (\det D^2 u_\e)^{-1}~&&\text{in}~ \Omega,\\\
u_\e ~&=\varphi~&&\text{on}~\p \Omega,\\\
w_\e ~&= \psi~&&\text{on}~\p \Omega,
\end{alignedat}
\right.
\end{equation}
where
\begin{multline}
\label{feq}
f_{\e}= \left\{\begin{array}{rl}
  \frac{\p F^0}{\p z}(x, u_\e(x)) -\displaystyle \sum_{i=1}^n \frac{\p}{\p x_i} \left(\gamma(x)[(|Du_\e|^2 +\delta(\e))^{\frac{q-2}{2}}u_{\e, x_i} -x_i]\right)&  \text{ if } x\in \Omega_0,\\[4pt]
\frac{1}{\e}\left(u_\e (x)-\varphi(x)- \e^{\frac{1}{3n^2}} (e^{\rho(x)}-1) \right) & \text{ if } x\in  \Omega\setminus \Omega_0.
\end{array}\right.
\end{multline}
The first two equations of (\ref{Abreuq}) arise as the Euler-Lagrange equation of the functional
  \begin{eqnarray}
  \label{Jfn2}
  J_{q,\e}(u)&:=&\int_{\Omega_0} \left[(|Du|^2 +\delta(\e))^{\frac{q}{2}}/q-x\cdot Du\right] \gamma(x)\,dx
+\int_{\Omega_0} F^0(x,u)\,dx\\&& -\e\int_{\Omega} \log \det D^2 u \,dx +\frac{1}{2\e} \int_{\Omega\backslash \Omega_0} (u-\varphi-\e^{\frac{1}{3n^2}} (e^{\rho(x)}-1) )^2\,dx\nonumber.
  \end{eqnarray}
Our final theorem, which is concerned with the solvability and asymptotic behavior of solutions to (\ref{Abreuq})-(\ref{feq}) when $\e\rightarrow 0$, states as follows.
\begin{thm}
\label{RCthm}
Let $\Omega_0$ and $\Omega$ be bounded, open, smooth, and convex domains in $\R^2$ ($n= 2$) where $\Omega$ is uniformly convex and contains $\overline{\Omega_0}$. 
Let $\varphi \in C^{5}(\overline{\Omega}), \psi\in C^{3}(\overline{\Omega})$ where $\varphi$ is convex, and $\inf_{\p\Omega}\psi>0$.
Assume that the smooth function $F^0$ satisfies (\ref{F0q}). Let $\gamma$ be a nonnegative and Lipschitz function on $\overline{\Omega}$. If $q>2$, then we also assume that $\gamma$ is a constant. 
If $\e>0$ is small, then, the following facts hold:
\begin{myindentpar}{1cm}
(i) The system (\ref{Abreuq})-(\ref{feq}) has  a uniformly convex solution $u_\e \in W^{4,s}(\Omega)$  for all $s\in (n,\infty)$.\\
(ii) Let $u_\e\in W^{4, s}(\Omega)$ $(s>n)$ be a solution to (\ref{Abreuq})-(\ref{feq}). Then, a subsequence of $u_\e$ converges uniformly on compact subsets of $\Omega$ to a minimizer $u\in \bar{S}[\varphi,\Omega_0]$  of (\ref{pb1}).  
\end{myindentpar}
\end{thm}
We will prove Theorem \ref{RCthm} in Section \ref{RCsect}.
 \begin{rem}
If $F^0(x, z)$ is uniformly convex with respect to $z$, then the minimizer of \eqref{pb1} is unique.
  When $q=2$, and $\delta(\e)=0$, the equation (\ref{Abreuq}) was considered in \cite{LeCPAM, Le7}.
  Suppose $1<q<2$ and $\delta(\e)=0$ in  (\ref{Abreuq}).
Even if we obtain
positive lower and upper bound for $\det D^2 u_\e$, the best regularity we can get for $u_\e$ is $C^{2,\alpha}(\overline{\Omega})$.
This is due to the jump over $\p\Omega_0$ of the terms on the right hand side. Thus, we cannot get $W^{4, s}(\Omega)$ solutions as stated in Theorem \ref{RCthm}.
\end{rem}

{\bf Notation.} The Legendre transform of $u$ will be denoted by $u^{\ast}$ while the partial Legendre transform of $u$ will be denoted by $u^{\star}$. We use $\nu$ to denote the unit outer normal to $\p\Omega$. 

The rest of the paper is organized as follows. The Legendre transform and partial Legendre transform and their applications to the Abreu equations will be discussed in Section \ref{Legsect}. In particular, we prove Theorem \ref{SBVn3} with the Legendre transform.  The proof of Theorem \ref{SBVq12} will be given in Section \ref{q12esect}. In Section \ref{RCsect}, we will prove Theorem \ref{RCthm}.


\section{Legendre transform and partial Legendre transform}
\label{Legsect}
\subsection{Legendre transform and regularity in general dimension}
In this section, we derive the dual equation of \eqref{4-eq-g} under Legendre transform in any dimension.
After the Legendre transform,  the equation is still a linearized Monge-Amp\`ere equation. Denote the Legendre transform $u^{\ast}$ of $u$ by $$u^{\ast}(y)=x \cdot Du-u, \quad \text{where }y=Du(x)\in \Omega^\ast =Du(\Omega). $$
\begin{prop}\label{new2}
Let $u\in C^4(\Omega)$ be a uniformly convex solution to \eqref{4-eq-g} in $\Omega$  where $f$ is given by (\ref{cond1}). Then in $\Omega^*=Du(\Omega)$, its Legendre transform $u^{\ast}$ satisfies
\begin{eqnarray}\label{new-eq-leg}
 u^{\ast, ij}D_{ij} w^*= f^*.
\end{eqnarray}
Here $( u^{\ast, ij})$ is the inverse matrix of $D^2u^{\ast}$, 
\begin{equation}
\label{det-leg}
 w^*= G((\det D^2 u^{\ast})^{-1})-(\det D^2 u^{\ast})^{-1} G'((\det D^2 u^{\ast})^{-1}),
\end{equation}
and
\begin{equation}
\label{righthand-leg}
 f^*=F_{p_ip_j}u^{\ast, ij}+(F_{p_i z} y_i+ F_{p_ix_i}-F_z),
\end{equation}
where 
$F_{p_ip_j}=F_{p_ip_j}(Du^{\ast}, y \cdot Du^{\ast}-u^{\ast}, y)$ and likewise for $F_{p_i z}$, $F_{p_ix_i}$, and $F_z$.

 If $u\in C^3(\Omega)$, then instead of (\ref{new-eq-leg}), we have
$$D_j((\det D^2 u^*)u^{\ast,ij}D_{i}w^*)  =f^*\det D^2u^*$$
in the weak sense.
\end{prop}
\begin{proof}
Recall that $u$ is a critical point of the functional
\begin{equation*}
  \mathcal{F}(u):=\int_\Omega F(x, u, Du)\, dx - \int_{\Omega} G (\det D^2 u) \,dx:=B(u)-A(u).
  \end{equation*}
From
$$\det D^2u=[\det D^2u^*]^{-1},\ \ dx=\det D^2u^{\ast}\,dy,$$
we have
\begin{eqnarray*}
\mathcal F(u)=\mathcal F^*( u^\ast):= B^*(u^{\ast}) -A^*(u^{\ast}),
\end{eqnarray*}
where
\begin{equation*}
A(u)=\int_{\Omega^*}\det D^2u^{\ast} G((\det D^2u^{\ast})^{-1})\,dy:=A^*(u^{\ast}),
\end{equation*}
and
\begin{eqnarray*}
B(u)&=&\int_{ \Omega^*}F(Du^{\ast}, y \cdot Du^{\ast}-u^{\ast}, y)\det D^2u^{\ast}\,dy:=B^*(u^{\ast}).
\end{eqnarray*}
Note that $u^{\ast}$ is a critical point of the dual functional $\mathcal{F}^*(u^\ast)$. To find an equation for $u^{\ast}$, we need to find the variations of $\mathcal{F}^{\ast}$.

Let $\varphi\in C^\infty_0(\Omega^*)$. Let $w^*$ be given as in \eqref{det-leg} and $(U^{\ast,ij})$ be the cofactor matrix of $D^2u^*$.  Note that $(U^{\ast, ij})$ is divergence free, that is
$$\sum_{i=1}^n D_{i} U^{\ast, ij}=0\quad \text{for all } j.$$
Then, integrating by parts twice, one finds
$$\frac{d A^*(u^*+t\varphi)}{dt}|_{t=0}=\int_{\Omega^*}U^{\ast,ij}w^*D_{ij}\varphi\,dy =\int_{\Omega^*}U^{\ast,ij}D_{ij}w^*\varphi\,dy.$$
By integration by parts, we find that
\begin{eqnarray*}
\frac{d B^*(u^*+t\varphi)}{dt}|_{t=0}&=&\int_{\Omega^*}F_{x_i}D_i \varphi \det D^2u^*\,dy+
\int_{\Omega^*}(y\cdot D\varphi-\varphi) F_z\det D^2u^*\,dy \nonumber\\[3pt]
&&+\int_{\Omega^*}FD_{ij}\varphi U^{\ast,ij}\,dy \nonumber\\[3pt]
&=&\int_{\Omega^*}F_{x_i}D_i \varphi \det D^2u^*\,dy+
\int_{\Omega^*}(y\cdot D\varphi-\varphi) F_z\det D^2u^*\,dy \nonumber\\[3pt]
&&-\int_{\Omega^*}\left[F_{x_k}D_{ki}u^* D_j\varphi U^{\ast,ij}\, + \p_{y_i}(y\cdot Du^*-u^*)F_z D_j\varphi U^{\ast,ij}\,+F_{p_i}D_j \varphi U^{\ast,ij}\right]\,dy.
\end{eqnarray*}
 Using $\p_{y_i}(y\cdot Du^*-u^*)=y_kD_{ki}u^*$ and $D_{ki}u^* U^{\ast,ij} =\det D^2 u^\ast \delta_{kj}$, we obtain
\begin{eqnarray*}
\frac{d B^*(u^*+t\varphi)}{dt}|_{t=0}
&=&
-\int_{\Omega^*}\varphi F_z\det D^2u^*\,dy-\int_{\Omega^*}F_{p_i}D_j \varphi U^{\ast,ij}\,dy\\[3pt]
&=&
-\int_{\Omega^*}\varphi F_z\det D^2u^*\,dy+\int_{\Omega^*}\varphi F_{p_ix_k} D_{kj}u^*U^{\ast,ij}\,dy\\[3pt]
&&+\int_{\Omega^*}\varphi F_{p_i z}  \p_{y_j}(y\cdot Du^*-u^*) U^{\ast,ij}\,dy +\int_{\Omega^*}\varphi F_{p_ip_j}U^{\ast,ij}\,dy\\[3pt]
&=&\int_{\Omega^*}\varphi [F_{p_ip_j}U^{\ast,ij}+(F_{p_i z} y_i+ F_{p_ix_i}-F_z)\det D^2u^* ]\,dy.
\end{eqnarray*}
Therefore
\begin{eqnarray*}
\frac{d \mathcal{F}^*(u^*+t\varphi)}{dt}|_{t=0} &=& \frac{d B^*(u^*+t\varphi)}{dt}|_{t=0} - \frac{d A^*(u^*+t\varphi)}{dt}|_{t=0}\\
&=& \int_{\Omega^*}\varphi [F_{p_ip_j}U^{\ast,ij}+(F_{p_i z} y_i+ F_{p_ix_i}-F_z)\det D^2u^* -U^{\ast,ij}D_{ij}w^*]\,dy.
\end{eqnarray*}
From $$\frac{d \mathcal{F}^*(u^*+t\varphi)}{dt}|_{t=0} =0,\quad\text{for all } \varphi\in C^\infty_0(\Omega^*),$$ we 
obtain
$$U^{\ast,ij}D_{ij}w^*= F_{p_ip_j}U^{\ast,ij}+(F_{p_i z} y_i+ F_{p_ix_i}-F_z)\det D^2u^* = f^* \det D^2 u^*$$
and this gives the desired equation for $u^{\ast}$.
 \end{proof}
 \begin{rem}
 A direct calculation as in Lemma 2.7 in \cite{Le2} gives another proof of Proposition \ref{new2}.
 \end{rem}

Using Proposition \ref{new2}, we can establish the interior higher order derivative estimates for the second boundary value problem of \eqref{Abreu}.

 \begin{proof}[Proof of Theorem \ref{SBVn3}]
 We use $C$ and $C_1$ to denote universal positive constants depending only on $\varphi, \psi, F^0,$ $n, q$ and $\Omega$.
  For $q>1$, we have from the convexity of $u$ that $$-\div (|Du|^{q-2} Du)\leq 0.$$
  Note that $F^0_z(x, z)\leq 0$.  Hence $U^{ij} D_{ij} w=-\div (|Du|^{q-2} Du) + F_z^0(x, u)\leq 0$. By the maximum principle applied to the equation $U^{ij} D_{ij} w\leq 0$, we see that $w$ attains its minimum value on the boundary. Thus $w\geq \inf_{\p\Omega}\psi>0$. This together with $\det D^2 u=w^{-1}$ gives a universal upper bound for $\det D^2 u$:
  $$\det D^2 u\leq C.$$
  Hence, from $u=\varphi$ on $\p\Omega$, we have $\sup_\Omega |u|\leq C$. Furthermore,  we can construct suitable barriers to get
  \begin{equation}
  \label{Du1}
  |Du|\leq C~\text{in } \Omega.
  \end{equation}
Let $u^{\ast} (y)$ be the Legendre transform of $u(x)$ where $y= Du(x)\in \Omega^\ast:= Du(\Omega)$. Then 
\begin{equation}
\label{uastbd}
|u^\ast|\leq C\quad\text{in }\Omega^\ast.
\end{equation}
Let $(U^{\ast, ij})$ be the cofactor matrix of $D^2 u^*$. Then, with the notation as in Proposition \ref{new2}, and $F(x, z, p)= |p|^q/q+F^0(x, z)$, we have
\begin{eqnarray*}
w^{\ast}&=&-\log \det D^2 u^\ast-1, \\ [4pt]
 f^\ast\det D^2 u^*&=&(F_{p_ip_j}u^{*,ij}-F_z)\det D^2 u^*\\[4pt]
&=&U^{*,ij}D_{ij}(|y|^q/q)-F^0_z\det D^2 u^*.
\end{eqnarray*}
  From (\ref{new-eq-leg}), we deduce that
   $u^{\ast}$ satisfies 
  \begin{eqnarray}
  \label{dual_eqn}
   D_j\left [U^{\ast, ij} D_{i}\left( |y|^q/q+ \log \det D^2u^{\ast} \right)\right] &=&U^{\ast, ij} D_{ij}\left( |y|^q/q+ \log \det D^2u^{\ast} \right)\nonumber \\
&=&F^0_z(Du^*, y\cdot Du^*-u^*)\det D^2 u^* 
  \end{eqnarray}
in $ \Omega^*$.  In view of (\ref{Du1}) and (\ref{uastbd}), we find that $F^0_z(Du^*, y\cdot Du^*-u^*)$ is universally bounded in $\Omega^\ast$. Hence, for a universally large constant $C_1>0$, we have in $\Omega^\ast$:
\begin{equation}
  \label{dual_eqn2}    D_j\left [U^{\ast, ij} D_{i}\left( |y|^q/q+ \log \det D^2u^{\ast} + C_1u^\ast\right)\right] 
=[F^0_z(Du^*, y\cdot Du^*-u^*)+ nC_1]\det D^2 u^* \geq  0.
\end{equation}

  If $y=Du(x)\in\p\Omega^*$, then $$\det D^2 u^*(y) = [\det D^2 u(x)]^{-1}= \psi(x) = \psi (Du^*(y)).$$
  This together with (\ref{Du1}) and (\ref{uastbd}) shows that on $\p \Omega^*$, $ |y|^q/q+ \log \det D^2u^{\ast} + C_1 u^\ast$ is bounded  by a universal constant. 
  We can apply the maximum principle to (\ref{dual_eqn}) to conclude that $$|y|^q/q+ \log \det D^2u^{\ast} + C_1 u^\ast \leq \sup_{\p\Omega^*}( |y|^q/q+ \log \det D^2u^{\ast} + C_1 u^\ast)\leq C\quad\text{in }\Omega^*.$$ Note that if $u\in C^{3}(\Omega)\cap C^2(\overline{\Omega})$, we apply the maximum principle for elliptic equations in divergence form (see \cite[Theorem 8.1]{GT}) to the divergence form of (\ref{dual_eqn2}). 
  
  In particular, $w(x)=\det D^2u^{\ast}(y)$ is bounded from above by a universal constant. Thus
  $\det D^2 u$ is bounded from below by a positive universal constant. In conclusion, we have
  \begin{equation}
  \label{det1}
  0<C^{-1}\leq \det D^2 u\leq C.
  \end{equation}
These bounds together with the the boundary data $\varphi$ of $u$ being $C^5({\overline{\Omega}})$ allow us to establish, from below, a universal (and positive) modulus of convexity of $u$ in the interior of $\Omega$; see \cite[Corollary 4.11 and Theorem 4.16]{F}.
  Now, we use the interior H\"older estimate for the linearized Monge-Amp\`ere equation with bounded right hand side \cite{CG, TW4}, applied to (\ref{dual_eqn}), to conclude that $\log \det D^2u^{\ast} $ is $C^{\alpha}$ in the interior of $\Omega^*$, for some $\alpha>0$ universal, with universal estimates.  This combined with the universal modulus of convexity of $u$ implies that $\det D^2 u$ is  $C^{\alpha}$ in the interior of $\Omega$.
  Therefore, from Caffarelli's $C^{2,\alpha}$ estimates for the Monge-Amp\`ere equation \cite{Ca1}, we obtain $C^{2,\alpha}$ estimates in the interior of $\Omega$ for $u$. Thus, in the interior of $\Omega$, $U^{ij} D_{ij}$ is a uniformly elliptic operator with $C^{\alpha}$ coefficients. Since $q>1$, $|Du|^{q-2} Du$ is a $C^{\beta}$ vector field in the interior of $\Omega$ for $\beta>0$ universal. Using \cite[Theorem 8.32]{GT}, we obtain from the first equation of (\ref{Abreu}), that is,
  $$ U^{ij} D_{ij} w= -\div(|Du|^{q-2} Du)+F^0_z(x, u),$$ 
the interior $C^{1,\gamma}$ estimates for $w$, where $\gamma:=\min\{\alpha, \beta\}$. This, in turns, gives the interior $C^{3,\gamma}$ estimates for $u$.
  
  When $q\geq 2$, we have better regularity estimates. In this case 
$$-\div (|Du|^{q-2}Du)=-|Du|^{q-2}\Delta u -(q-2)|Du|^{q-4} D_i u D_j u D_{ij} u$$
  is $C^{\alpha_q}$ in the interior of $\Omega$ for some $\alpha_q= \alpha_q(\alpha, q)>0$ universal.
 Now we can use the standard Schauder theory to the first equation of (\ref{Abreu}) to get the interior $C^{2,\alpha_q}$ estimates for $w$. Hence, we get the interior $C^{4,\alpha_q}$ estimates for $u$.
  \end{proof}

We also derive global smoothness estimates for the second boundary value problem of  (\ref{4-eq-g}) in terms of the $W^{2, n}(\Omega)$ norm of the solutions when $F$ is of a special form.
\begin{prop}
\label{W2nprop}
Let $\Omega\subset\R^n$ be an open, smooth, bounded and uniformly convex domain.  
Assume that $\varphi\in C^{5}(\overline{\Omega})$ and $\psi\in C^{3}(\overline{\Omega})$ with $\inf_{\p \Omega}\psi>0$.
Suppose $F(x, z, p)= F^0(x, z) + F^1(p)$ is smooth with $|D^2 F^1(p)|\leq M$ for all $p=(p_1, \cdots, p_n)\in \R^n.$ Consider a smooth solution $u$ to the second boundary value problem for (\ref{4-eq-g}) where $G(t)=\log t$. 
Assume that $ \|u\|_{W^{2, n}(\Omega)}\leq K$.
Then $u\in C^{4,\alpha}(\overline{\Omega})$ with
$$ \|u\|_{C^{4, \alpha}(\overline{\Omega})}\leq C$$
where $\alpha>0$ and $C$ depends on $F^0, F^1, K, M,\varphi,\psi, n$ and $\Omega$.
\end{prop}
\begin{proof}
We use $C, C_1, C_2, \cdots$ to denote universal positive constants depending only on $F^0, F^1$, $K, M$, $\varphi,\psi$, $n$ and $\Omega$.
When $F(x, z, p)= F^0(x, z) + F^1(p)$, we have
\begin{equation}\label{fform} f(x) = F^0_z(x, u(x))- F^1_{p_i p_j}( Du(x)) D_{ij} u(x)
\end{equation}
and
\begin{equation}
\label{fLn}
\|f\|_{L^n(\Omega)} \leq C(n, M, K, F^0,\Omega).
\end{equation}
Note that $$\det (U^{ij}) = (\det D^2 u)^{n-1}= w^{-(n-1)}.$$ 
We apply the Aleksandrov-Bakelman-Pucci estimate (see, \cite[Theorem 9.1]{GT}) to $U^{ij}D_{ij}w=f$ in $\Omega$ with
$w=\psi$ on $\p\Omega$ to find that
\begin{equation*}\sup_{\Omega} w \leq \sup_{\partial \Omega} \psi +C(n,\Omega) \left\|\frac{f}{(\det (U^{ij}))^{1/n}}\right\|_{L^n(\Omega)}
\leq \sup_{\partial \Omega} \psi +C(n,\Omega) \left\| f\right \|_{L^n(\Omega)} \sup_{\Omega} (w^{(n-1)/n}).
\end{equation*}
It follows that 
$w\leq C$ and hence 
\begin{equation*}w\leq C,\quad \det D^2 u\geq C^{-1}>0.
\end{equation*}
Using the above estimates and arguing as in \cite[Lemma 2.5]{Le2}, we have
\begin{equation}
\label{Du2}
\sup_{\Omega}|Du|\leq C.
\end{equation}
We use the Legendre transform and notation as in Proposition \ref{new2}. Then 
$$ w^{\ast}=-\log \det D^2 u^\ast-1, \  F_{p_ip_j}u^{*,ij}=u^{*,ij}D_{ij}F^1(y) $$
and, from (\ref{righthand-leg}), we deduce that
\beq
\label{123}
 u^{\ast, ij} D_{ij}(w^{\ast}- F^1(y))=-F^{0}_z(Du^{\ast}, y \cdot Du^{\ast}-u^{\ast})\quad \text{in } \Omega^*.
\eeq
From (\ref{Du2}) and  $u^*(y) = x\cdot Du(x)-u(x) $ where $y=Du(x)$, we deduce
$$|u^*| + |F^{0}_z(Du^{\ast}, y \cdot Du^{\ast}-u^{\ast})|\leq C_1 \quad \text{in } \Omega^*.$$
Thus, for a large universal constant $C_2>0$, we have
$$u^{\ast, ij} D_{ij}(w^{\ast}- F^1(y)+ C_2 u^*) =-F^{0}(Du^{\ast}, y \cdot Du^{\ast}-u^{\ast}) + nC_2>0\quad \text{in } \Omega^*.$$
Hence, by the maximum principle, $w^{\ast}- F^1(y)+ C_2 u^*$ attains it maximum on $\p \Omega^*$. If $y=Du(x)\in \p\Omega^*$, then 
$$w^{\ast}(y)=-\log \det D^2 u^\ast(y)-1= \log \det D^2 u(x) -1 =\log \frac{1}{\psi (Du^*(y))}-1\leq C.$$
From this, we find that $$w^{\ast}\leq C \quad \text{in } \Omega^*.$$ 
Therefore, 
$\det D^2 u(x)= e^{w^*(Du(x))+1}\leq C$ in $\Omega$. This combined with the lower bound for $\det D^2 u$ gives
\begin{equation}
\label{detCC}
0<C^{-1}\leq \det D^2 u\leq C  \quad \text{in } \Omega.
\end{equation}
Now, using (\ref{detCC}) and (\ref{fLn}), we can apply the global $C^{\alpha}$ estimates for the linearized Monge-Amp\`ere equation (see, \cite[Theorem 1.4]{Le1}) to 
$$U^{ij} D_{ij} w= f  \quad \text{in } \Omega, ~\quad w=\psi  \quad \text{on } \p \Omega,$$
to get
$$\|w\|_{C^{\alpha}(\overline{\Omega})}\leq C_3$$
where $\alpha$ and $C_3$ are universal positive constants. From the global $C^{2,\alpha}$ estimates for the Monge-Amp\`ere equation  (see  \cite{TW3}) applied to
 \begin{equation}
 \label{wCC}
 \left\{
 \begin{alignedat}{2}
   \det D^2 u~&=w^{-1}~&&\text{in} ~\Omega, \\[4pt]
 u&= \varphi~&&\text{on}~ \p\Omega,
 \end{alignedat}
 \right.
\end{equation}
we find
$$\|u\|_{C^{2, \alpha}(\overline{\Omega})}\leq C_4.$$
Therefore, the second order operator $U^{ij} D_{ij}$ is uniformly elliptic with $C^{\alpha}(\overline{\Omega})$ coefficients. Moreover, from (\ref{fform}), we find that $f\in C^{\alpha}(\overline{\Omega})$. Using the classical Schauder estimates to $U^{ij} D_{ij} w= f$, we deduce that $w\in C^{2,\alpha}(\overline{\Omega})$ with 
$\|w\|_{C^{2, \alpha}(\overline{\Omega})}\leq C.$
With this estimate, (\ref{wCC}) easily gives
$$\|u\|_{C^{4, \alpha}(\overline{\Omega})}\leq C.$$
\end{proof}

\begin{rem}
The assumption $|F^1_{p_i p_j}(x, p)|\leq M$ in Proposition \ref{W2nprop} can be removed if one has $\|u\|_{C^{0,1}(\Omega)}+\|u\|_{W^{2, n}(\Omega)}\leq K$ or $\|u\|_{W^{2, n+\e}(\Omega)}\leq K$ for some $\e>0$.
\end{rem}

From (\ref{123}), we also obtain the following interior estimates.
 
\begin{prop}\label{int-est-high}
Let $\Omega\subset\mathbb R^n$ be a convex domain. Let $F^0: \R^n\times \R\rightarrow\R$ be a smooth function. Let
$u\in C^3(\Omega)$ be a uniformly convex  solution to
\beq\label{121}
\sum_{i, j=1}^{n}U^{ij}D_{ij}[(\det D^2 u)^{-1}] =-\div (|Du|^{q-2} Du)+F^0_z(x,u)\quad\text{in }\Omega
\eeq
that satisfies 
\beq\label{detcond}
0<\lambda<\det D^2u\leq \Lambda.
\eeq 
Then there exists a constant $\alpha\in (0, 1)$ depending only on $\lambda,\Lambda, n$ and $q$ with the following property: For any $\Omega'\Subset\Omega$, there exists a constant $C>0$ depending on $\sup_\Omega|u|$,  the modulus of convexity of $u$, $\lambda$, $\Lambda$, $n$, $q$, $F^0$ and $dist(\Omega',\partial\Omega)$, such that 
$$\|u\|_{C^{4, \alpha}(\Omega')}\leq C~\quad\text{if } q\geq 2$$
and
$$\|u\|_{C^{3, \alpha}(\Omega')}\leq C~\quad\text{if } 1<q< 2.$$
\end{prop}

\begin{rem} 
In dimension two, \eqref{detcond} implies a positive lower bound on the modulus of convexity of $u$; see, for example \cite[Lemma 2.5]{Li}.
\end{rem}

\subsection{Partial Legendre transform in two dimensions}

In this section, we consider $n=2$ and write $u(x)=u(x_1,x_2)$.
The partial Legendre transform in the $x_1$-variable is 
\beq\label{p-leg}
u^\star(\xi, \eta)=x_1u_{x_1}(x_1, x_2)-u(x_1,x_2),
\eeq
where $$y=(\xi, \eta)=\mathcal P(x_1, x_2):=(u_{x_1}, x_2)\in \mathcal{P}(\Omega):=\Omega^\star.$$
We have
$$
\frac{\partial(\xi,\eta)}{\partial(x_1,x_2)}=
\begin{pmatrix}
u_{x_1 x_1}  &\ \ u_{x_1 x_2}\\[4pt]
 0 &\ \ 1\\
\end{pmatrix}, \quad\text{and } 
\
 \frac{\partial(x_1,x_2)}{\partial(\xi,\eta)}=
\begin{pmatrix}
\frac{1}{u_{x_1 x_1}}  &\ \  -\frac{u_{x_1 x_2}}{u_{x_1 x_1}}\\[4pt]
0 &\ \ 1\\
\end{pmatrix}.
$$
Hence, 
$$u^\star_\xi=x_1,\ u^\star_{\xi\xi}=\frac{1}{u_{x_1 x_1}},\  u^\star_\eta=-u_{x_2},\ u^\star_{\eta\eta}=-\frac{\det D^2u}{u_{x_1 x_1}},\ u^\star_{\xi\eta}=-\frac{u_{x_1 x_2}}{u_{x_1 x_1}}.$$
In the following proposition, we deduce the dual equation for \eqref{4-eq-g} under partial Legendre transform. One can derive the dual equation for the general case of $F(x,z, p)$ and $G$, but for simplicity we only consider a special case for
 \eqref{4-eq-g} which is appropriate for the proof of Theorem \ref{SBVq12}. This is the case of equation (\ref{SBVq12e}) in Section \ref{q12esect}. As we explained in the introduction, when $1<q<2$, we can not expect $C^4$ solution to \eqref{Abreu4}.
\begin{prop}\label{new-eq}
Let $G(t)=\log t$ and $F(x, z, p) = (|p|^2 +\delta)^{\frac{q}{2}}/q + F^0(x, z)$, where $\delta\geq 0$ and $q>1$.
Let $u\in C^4(\Omega)$ be a uniformly convex solution to \eqref{4-eq-g} in $\Omega$  where $f$ is given by (\ref{cond1}). Then in $\Omega^\star=\mathcal P(\Omega)$, its partial Legendre transform $u^\star$ satisfies
\begin{eqnarray}
\label{eq-new2}
w^\star w^\star_{\xi\xi}+w^\star_{\eta\eta}-{w^\star_\xi}^2-\frac{2}{w^\star}{w^\star_\eta}^2={w^\star}^2f^\star. 
\end{eqnarray}
Here $w^\star=-\frac{u^\star_{\eta\eta}}{u^\star_{\xi\xi}}$ and 
\beq\label{righthand}
f^\star=
(\xi^2+{u^\star_\eta }^2+\delta)^{\frac{q}{2}-1}(1+{u^\star_{\xi\eta}}^2-  u^\star_{\eta\eta}u^\star_{\xi\xi})\nonumber\\
+(q-2)(\xi^2+{u^\star_\eta }^2+\delta)^{\frac{q}{2}-2}[(\xi+u^\star_\eta u^\star_{\xi\eta})^2-{u^\star_\eta}^2u^\star_{\xi\xi}u^\star_{\eta\eta}]-F^0_zu^\star_{\xi\xi},
\eeq
where $F^0_z=F^0_z(u^\star_\xi, \eta,\xi u^\star_\xi-u^\star)$.
\end{prop}

\begin{proof}
Recall that $u$ is a critical point of the functional
\begin{eqnarray*}
  \mathcal{F}(u)&:=&\int_{\Omega} \frac{(|Du|^2 +\delta)^{\frac{q}{2}}}{q}\,dx - \int_{\Omega}\log \det D^2 u \,dx+\int_\Omega F^0(x,u)\,dx\\
&:=&\calF^1(u)-A(u)+\calF^0(u).
  \end{eqnarray*}
From
$$\det D^2u=-\frac{u^\star_{\eta\eta}}{u^\star_{\xi\xi}},\ \ dx_1 dx_2=u^\star_{\xi\xi}\,d\xi d\eta,$$
we have
\begin{eqnarray*}
\calF^1(u)&=&\int_{\Omega^\star} \frac{(\xi^2+{u^\star_\eta}^2+\delta)^{\frac{q}{2}}}{q}u^\star_{\xi\xi}\,d\xi d\eta:=\calF^{1,\star}(u^\star), \\[4pt]
A(u)
&=&\int_{\Omega^\star}\log\left(-\frac{u^\star_{\eta\eta}}{u^\star_{\xi\xi}}\right) u^\star_{\xi\xi}\,d\xi d\eta:=A^\star(u^\star),\\[4pt]
\calF^0(u)&=&\int_\Omega F^0(x_1, x_2, u)\, dx_1 dx_2 \\
&=&\int_{\Omega^\star}F^0(u^\star_\xi, \eta,\xi u^\star_\xi-u^\star)u^\star_{\xi\xi}\,d\xi d\eta:=\calF^{0, \star}(u^\star).
\end{eqnarray*}
Note that $u^{\star}$ is a critical point of the dual functional $$\calF^\star(u^\star)=\calF^{1,\star}(u^\star)-A^\star(u^\star)+ \calF^{0, \star}(u^\star).$$ To find an equation for $u^{\star}$, we need to find the variations of $\calF^{\star}$.

Let $\varphi\in C^\infty_0(\Omega^\star)$. Using integration by parts, we obtain
\begin{eqnarray*}
\frac{d \calF^{1, \star}(u^\star+t\varphi)}{dt}|_{t=0}&=&\int_{\Omega^\star}\frac{(\xi^2+{u^\star_\eta}^2+\delta)^{\frac{q}{2}}}{q}\varphi_{\xi\xi}\,d\xi d\eta+\int_{\Omega^\star} (\xi^2+{u^\star_\eta}^2+\delta)^{\frac{q}{2}-1}u^\star_\eta \varphi_\eta u^\star_{\xi\xi}\,d\xi d\eta\\
&=&-\int_{\Omega^\star}(\xi^2+{u^\star_\eta}^2+\delta)^{\frac{q}{2}-1}(\xi+u^\star_\eta u^\star_{\xi\eta})\varphi_{\xi}\,d\xi d\eta\\&&+\int_{\Omega^\star} (\xi^2+{u^\star_\eta}^2+\delta)^{\frac{q}{2}-1}u^\star_\eta \varphi_\eta u^\star_{\xi\xi}\,d\xi d\eta\\
&=&\int_{\Omega^\star}[(\xi^2+{u^\star_\eta}^2+\delta)^{\frac{q}{2}-1}(\xi+u^\star_\eta u^\star_{\xi\eta})]_{\xi}\varphi\,d\xi d\eta\\&&-\int_{\Omega^\star} [(\xi^2+{u^\star_\eta}^2 +\delta)^{\frac{q}{2}-1}u^\star_\eta  u^\star_{\xi\xi}]_\eta\varphi\,d\xi d\eta.
\end{eqnarray*}

\begin{eqnarray*}
\frac{d A^\star(u^\star+t\varphi)}{dt}|_{t=0}&=&\int_{\Omega^\star}-\frac{u^\star_{\xi\xi}}{u^\star_{\eta\eta}} \left(-\frac{\varphi_{\eta\eta}u^\star_{\xi\xi}-u^\star_{\eta\eta}\varphi_{\xi\xi}}{{u^\star_{\xi\xi}}^2}\right)u^\star_{\xi\xi}+\log\left(-\frac{u^\star_{\eta\eta}}{u^\star_{\xi\xi}}\right) \varphi_{\xi\xi}\,d\xi d\eta\\[3pt]
&=&\int_{\Omega^\star} \frac{\varphi_{\eta\eta}u^\star_{\xi\xi}-u^\star_{\eta\eta}\varphi_{\xi\xi}}{u^\star_{\eta\eta}}+\log\left(-\frac{u^\star_{\eta\eta}}{u^\star_{\xi\xi}}\right) \varphi_{\xi\xi}\,d\xi d\eta\\
&=& \int_{\Omega^\star} -(w^\star)^{-1} \varphi_{\eta\eta} -\varphi_{\xi\xi} +\log w^{\star} \varphi_{\xi\xi}\,d\xi d\eta\\
&=&  \int_{\Omega^\star} \left(-[(w^\star)^{-1}]_{\eta\eta}+ (\log w^\star)_{\xi\xi}\right)\varphi \,d\xi d\eta.
\end{eqnarray*}
\begin{eqnarray*}
\frac{d \calF^{0, \star}(u^\star+t\varphi)}{dt}|_{t=0}&=&
\int_{\Omega^\star}(\xi\varphi_\xi-\varphi) F^0_z u^\star_{\xi\xi}\,d\xi d\eta+\int_{\Omega^\star}F^0_{x_1}\varphi_\xi u^\star_{\xi\xi}\,d\xi d\eta+\int_{\Omega^\star}F^0\varphi_{\xi\xi}\,d\xi d\eta\nonumber\\[3pt]
&=&
\int_{\Omega^\star}(\xi\varphi_\xi-\varphi) F^0_z u^\star_{\xi\xi}\,d\xi d\eta-\int_{\Omega^\star}(\xi u^\star_\xi-u^\star)_\xi F^0_z\varphi_{\xi}\,d\xi d\eta\nonumber\\[3pt]
&=&-\int_{\Omega^\star}F^0_zu^\star_{\xi\xi}\varphi\,d\xi d\eta\nonumber.
\end{eqnarray*}
Therefore 
\begin{eqnarray*}
\frac{d \calF^\star(u^\star+t\varphi)}{dt}|_{t=0}&=&\frac{d \calF^{1, \star}(u^\star+t\varphi)}{dt}|_{t=0} -\frac{d A^\star(u^\star+t\varphi)}{dt}|_{t=0} + \frac{d \calF^{0, \star}(u^\star+t\varphi)}{dt}|_{t=0}\\
&=& \int_{\Omega^\star}\left([(\xi^2+{u^\star_\eta}^2+\delta)^{\frac{q}{2}-1}(\xi+u^\star_\eta u^\star_{\xi\eta})]_{\xi}- [(\xi^2+{u^\star_\eta}^2+\delta)^{\frac{q}{2}-1}u^\star_\eta  u^\star_{\xi\xi}]_\eta\right)\varphi\,d\xi d\eta\\
&&-  \int_{\Omega^\star} \left(-[(w^\star)^{-1}]_{\eta\eta}+ (\log w^\star)_{\xi\xi}+ F^0_zu^\star_{\xi\xi}\right)\varphi \,d\xi d\eta.
\end{eqnarray*}
From $$\frac{d \mathcal{F}^\star(u^\star+t\varphi)}{dt}|_{t=0} =0,\quad\text{for all } \varphi\in C^\infty_0(\Omega^\star),$$ we 
find that, after the partial Legendre transformation, equation (\ref{4-eq-g}) becomes
\begin{eqnarray*}
-[(w^\star)^{-1}]_{\eta\eta}+ (\log w^\star)_{\xi\xi}&=&
(\xi^2+{u^\star_\eta}^2+\delta)^{\frac{q}{2}-1}(1+{u^\star_{\xi\eta}}^2-  u^\star_{\eta\eta}u^\star_{\xi\xi})\nonumber\\
&&+(q-2)(\xi^2+{u^\star_\eta}^2+\delta)^{\frac{q}{2}-2}[(\xi+u^\star_\eta u^\star_{\xi\eta})^2-{u^\star_\eta}^2u^\star_{\xi\xi}u^\star_{\eta\eta}]-F^0_z u^\star_{\xi\xi}.
\end{eqnarray*}
After simplifications, it becomes
\begin{eqnarray}
w^\star w^\star_{\xi\xi}+w^\star_{\eta\eta}-{w^\star_\xi}^2-\frac{2}{w^\star}{w^\star_\eta}^2&=&{w^\star}^2
\cdot\{ (\xi^2+{u^\star_\eta}^2+\delta)^{\frac{q}{2}-1}(1+{u^\star_{\xi\eta}}^2-  u^\star_{\eta\eta}u^\star_{\xi\xi})\nonumber\\
&&+(q-2)(\xi^2+{u^\star_\eta}^2+\delta)^{\frac{q}{2}-2}[(\xi+u^\star_\eta u^\star_{\xi\eta})^2-{u^\star_\eta}^2u^\star_{\xi\xi}u^\star_{\eta\eta}]-F^0_z u^\star_{\xi\xi}\}.
\end{eqnarray}
Therefore, (\ref{eq-new2}) is proved.
\end{proof}


\section{Proof of Theorem \ref{SBVq12}}
\label{q12esect}
\begin{proof}[Proof of Theorem \ref{SBVq12}]
Theorem \ref{SBVq12} follows from Theorem \ref{SBVq12e} by letting $\e\rightarrow 0$.
\end{proof}
\begin{thm}
\label{SBVq12e}
Let $\Omega\subset\R^2$ be an open, smooth, bounded and uniformly convex domain.  Let $q>1$. Assume (\ref{Qcond}) holds.
Assume that $\varphi\in C^{5}(\overline{\Omega})$ and $\psi\in C^{3}(\overline{\Omega})$ with $\inf_{\p \Omega}\psi>0$. Let $\delta(\e)=\e$.
Consider  the following
second boundary value problem:
\begin{equation}
\label{Abreu4e}
  \left\{ 
  \begin{alignedat}{2}\sum_{i, j=1}^{2}U^{ij}D_{ij}w~& =-div ((|Du|^2 +\delta(\e))^{\frac{q-2}{2}} Du) + F^0_z(x, u)~&&\text{in} ~\Omega, \\\
 w~&= (\det D^2 u)^{-1}~&&\text{in}~ \Omega,\\\
u ~&=\varphi~&&\text{on}~\p \Omega,\\\
w ~&= \psi~&&\text{on}~\p \Omega.
\end{alignedat}
\right.
\end{equation}
Then, the following facts hold:
\begin{myindentpar}{1cm}
(a) There exists a  uniformly convex solution $u_\e\in C^{4, \gamma}(\overline{\Omega})$ to (\ref{Abreu4e}) for all $\gamma\in (0,1)$. 
\\
(b) If $q\geq 2$, then 
$$\|u_\e\|_{C^{4,\beta}(\overline{\Omega})}\leq C$$
for some $\beta\in (0, 1)$ and $C>0$ depending on $q,\Omega$, $\omega$, $F^0$, $\varphi$ and $\psi$. \\
(c) If $1<q<2$, then
$$\|u_\e\|_{C^{3,\beta}(\overline{\Omega})}\leq C$$
for some $\beta\in (0, 1)$ and $C>0$ depending on $q,\Omega$, $\omega$, $F^0$, $\varphi$ and $\psi$.
\end{myindentpar}
\end{thm}

The proof of Theorem \ref{SBVq12e}, using a priori estimates and degree theory, is similar to that of Theorem 2.1 in \cite{LeCPAM}. 
We focus here on the a priori estimates. 

{\it For the rest of this section, let $u_\e$ be a smooth, uniformly convex solution to (\ref{Abreu4e}).
We drop the subscript $\e$ in $u_\e$, and $w_\e$, etc. to simplify notations. Universal constants in the following paragraphs depend only on $\varphi, \psi$, $\inf_{\p\Omega} \psi$, $\Omega$, $q$, $F^0$, and $\omega$. However, they are independent of $\e$.}
\subsection{Uniform bound for $u$}
We first establish the universal bound for $u$. 
 \begin{lem}
 \label{umaxlem}
 There is a universal constant $C_1$ such that
 $$\|u\|_{L^{\infty}(\Omega)}\leq C_1.$$
 \end{lem}

 \begin{proof}
 Note that, for a convex function $u\in C^2(\overline{\Omega})$ with $u=\varphi$ on $\p\Omega$, we have (see, e.g., \cite[inequality (2.7)]{Le2})
\begin{equation*}\|u\|_{L^{\infty}(\Omega)} \leq C \|\varphi\|_{L^{\infty}(\Omega)}+ C(n,\Omega,  \|\varphi\|_{C^2(\Omega)})\left(
\int_{\p \Omega}  (u_{\nu}^+)^n \right)^{1/n}\quad \text{where } u_{\nu}^+ =\max (0, u_{\nu}).
\end{equation*}
Thus, to prove the lemma, it suffices to prove 
 \begin{equation}
 \label{unudS}
 \int_{\p\Omega} (u_\nu^{+})^n \,dS\leq C.
 \end{equation}
For this, we use the arguments as in the proof of \cite[Lemma 4.2]{LeCPAM}.
Let $\rho$ be a strictly convex defining function of $\Omega$ as in (\ref{rhoeq}). Let $$\tilde u= \varphi + \mu (e^{\rho}-1).$$ 
Then, for $\mu$ universally large, depending on $n,\Omega$ and  $\|\varphi\|_{C^2(\Omega)}$, the function $\tilde u$ is convex, belongs to $C^{5}(\overline{\Omega})$. Furthermore, as in \cite[Lemma 2.1]{Le2}, we can verify that for some constant $C$ depending only on $n$, $\Omega$, and $\| \varphi\|_{C^{4}(\overline{\Omega})}$
 \begin{myindentpar}{1cm}
  (i) $
\| \tilde{u} \|_{C^{4}(\overline{\Omega})} \leq C,\quad \textrm{and } \det D^2\tilde{u} \ge C^{-1}>0,
$\\
(ii) letting $\tilde{w}=[\det D^2 \tilde{u}]^{-1}$, and denoting by $(\tilde{U}^{ij})$ the 
cofactor matrix of $D^2\tilde u$, then  $$\left\|\tilde U^{ij}D_{ij}\tilde w\right\|_{L^{\infty}(\Omega)}\leq C.$$
  \end{myindentpar}
Let $K(x)$ denote the Gauss curvature of $\p\Omega$ at $x\in\p\Omega$. Then, from the estimate (4.10) in the proof of \cite[Lemma 4.2]{LeCPAM} with $\theta=0$ and $n=2$ which uses (i) and (ii), we obtain
\begin{eqnarray}
\label{concaveq12}
 \int_{\partial \Omega} K \psi u_{\nu}^2 \,dS
& \leq &   \int_{\Omega} [-F^0_z(x, u) + \div ((|Du|^2+\delta(\e))^{\frac{q-2}{2}} Du)] (u-\tilde u)\,dx \nonumber\\[4pt]
&&  +C\left(
\int_{\p \Omega}  (u_{\nu}^+)^2 \,dS \right)^{1/2} + C.
\end{eqnarray}
Since $\tilde u$ is universally bounded, we use (\ref{Qcond}) to get a universal constant $C>0$ such that 
$$-F^0_z(x, u)(u-\tilde u)\leq C.$$
From integration by parts and (i), we find
\begin{eqnarray*}
\int_{\Omega} \div ((|Du|^2+\delta(\e))^{\frac{q-2}{2}}Du) (u-\tilde u) \,dx&=& \int_{\p\Omega} - (|Du|^2+\delta(\e))^{\frac{q-2}{2}}Du\cdot (Du-D\tilde u)\,dx\\[4pt]
&\leq& C(q,  \sup_{\overline{\Omega}}|D\tilde u|)\leq C.
\end{eqnarray*}
Thus (\ref{concaveq12})  gives
\begin{equation*}
 \int_{\partial \Omega} K \psi u_{\nu}^2 \,dS
 \leq  C + C\left(
\int_{\p \Omega}  (u_{\nu}^+)^2 \,dS\right)^{1/2}.
\end{equation*}
In view of  $\inf_{\p\Omega} (K\psi)>0$, we deduce that $\int_{\p\Omega} u_\nu^2 \,dS\leq C$ which establishes (\ref{unudS}).
\end{proof}
\subsection{Hessian determinant bounds for $u$}
Next we provide a universal upper bound for $\det D^2 u$.
\begin{lem}
\label{updet_lem}
 There is a universal constant $C_2$ such that
 $$\det D^2 u\leq C_2\quad\text{in } \Omega.$$
 \end{lem}
\begin{proof}
By Lemma \ref{umaxlem}, $\sup_{\Omega} |u|$ is universally bounded by a constant $C_1$.  We will use the partial Legendre transform in Proposition \ref{new-eq}. Note that
\begin{eqnarray*}N(u^\star):
&=&(\xi^2+{u^\star_\eta}^2 +\delta)(1+{u^\star_{\xi\eta}}^2-  u^\star_{\eta\eta}u^\star_{\xi\xi})+(q-2)[(\xi+u^\star_\eta u^\star_{\xi\eta})^2-{u^\star_\eta}^2u^\star_{\xi\xi}u^\star_{\eta\eta}]\\
&\geq &(q-1)\xi^2+{u^*_\eta}^2+\xi^2{u^*_{\xi\eta}}^2+(q-1){u^\star_\eta}^2{u^\star_{\xi\eta}}^2+2(q-2){u^\star_\eta}\xi {u^\star_{\xi\eta}}\\&&-\xi^2 u^\star_{\xi\xi}u^\star_{\eta\eta} -(q-1){u^\star_\eta}^2u^\star_{\xi\xi}u^\star_{\eta\eta}\geq 0.
\end{eqnarray*}
Thus, recalling (\ref{Qcond}), we find
\begin{eqnarray}
\label{inequ-1}
f^\star=(\xi^2+{u^\star_\eta}^2+\delta)^{\frac{q}{2}-2} N(u^\star) -F^0_z u^\star_{\xi\xi} \geq -F^0_z u^\star_{\xi\xi}\geq-Cu^\star_{\xi\xi}
\end{eqnarray}
where $C$ depends on $C_1$ and $\omega$.

Let $Z=\log w^\star+\alpha (\xi u^\star_\xi-u^\star)$. Note that $\xi u^\star_\xi-u^\star=u$ is universally bounded by Lemma \ref{umaxlem}.
By simple computations,
\begin{eqnarray*}
(\xi u^\star_\xi-u^\star)_\xi&=&\xi u^\star_{\xi\xi},\\
(\xi u^\star_\xi-u^\star)_\eta&=&\xi u^\star_{\xi\eta}-u^\star_\eta,\\
(\xi u^\star_\xi-u^\star)_{\xi\xi}&=& u^\star_{\xi\xi}+\xi u^\star_{\xi\xi\xi},\\
(\xi u^\star_\xi-u^\star)_{\eta\eta}&=&\xi u^\star_{\xi\eta\eta}-u^\star_{\eta\eta},\\
w^\star_\xi&=&-\frac{u^\star_{\eta\eta\xi}u^\star_{\xi\xi}-u^\star_{\eta\eta}u^\star_{\xi\xi\xi}}{{u^\star_{\xi\xi}}^2}=-\frac{w^\star u^\star_{\xi\xi\xi}+u^\star_{\xi\eta\eta}}{u^\star_{\xi\xi}}.
\end{eqnarray*}
So we have
\begin{eqnarray*}
Z_\xi&=&\frac{w^\star_\xi}{w^\star}+\alpha\xi u^\star_{\xi\xi},\\
Z_{\xi\xi}&=&\frac{w^\star_{\xi\xi}}{w^\star}-\frac{{w^\star_\xi}^2}{{w^\star}^2}+\alpha(u^\star_{\xi\xi}+\xi u^\star_{\xi\xi\xi}),\\
Z_\eta&=&\frac{w^\star_\eta}{w^\star}+\alpha(\xi u^\star_{\xi\eta}-u^\star_\eta),\\
Z_{\eta\eta}&=&\frac{w^\star_{\eta\eta}}{w^\star}-\frac{{w^\star_\eta}^2}{{w^\star}^2}+\alpha(\xi u^\star_{\xi\eta\eta}-u^\star_{\eta\eta}).
\end{eqnarray*}
Then, using (\ref{inequ-1}), we can estimate 
\begin{eqnarray*}
&&w^\star Z_{\xi\xi}+Z_{\eta\eta}\\
&=&\frac{1}{w^\star}(w^\star w^\star_{\xi\xi}+w^\star_{\eta\eta}-{w^\star_\xi}^2-\frac{1}{w^\star}{w^\star_\eta}^2)+\alpha\xi(w^\star u^\star_{\xi\xi\xi}+u^\star_{\xi\eta\eta})+\alpha(w^\star u^\star_{\xi\xi}-u^\star_{\eta\eta})\\
&= &\frac{1}{w^\star}(\frac{1}{w^\star}{w^\star_\eta}^2+{w^\star}^2 f^\star)-\alpha\xi w^\star_\xi u^\star_{\xi\xi}+2\alpha w^\star u^\star_{\xi\xi}\\
&\geq &\frac{1}{w^\star}(\frac{1}{w^\star}{w^\star_\eta}^2-C{w^\star}^2u^\star_{\xi\xi})-\alpha\xi w^\star_\xi u^\star_{\xi\xi}+2\alpha w^\star u^\star_{\xi\xi}\\
&=&\frac{{w^\star_\eta}^2}{{w^\star}^2}-\alpha \xi w^\star Z_\xi+\alpha^2\xi^2 w^\star u^\star_{\xi\xi}+(2\alpha-C) w^\star u^\star_{\xi\xi}.
\end{eqnarray*}
Choosing $\alpha>0$ suffienctly large, we have  $$w^\star  Z_{\xi\xi}+Z_{\eta\eta}+\alpha \xi w^\star Z_\xi\geq 0.$$
Hence $Z=\log w^\star+\alpha (\xi u^\star_\xi-u^\star)$ attains its maximum on $\partial\Omega^\star$. Note that if $y= (\xi,\eta)= \mathcal{P}(x)\in \p\Omega^\star$, then 
$$w^\star(y) =-\frac{u^\star_{\eta\eta}}{u^\star_{\xi\xi}}= \det D^2 u(x) = \frac{1}{\psi(x)} \leq \frac{1}{\inf_{\p\Omega}\psi}\leq C.$$
It follows that $w^\star \leq C$, and hence, $\det D^2u\leq C_2$.
\end{proof}
\begin{rem}
The above calculations use that $w_\e^{\star}\in C^2$. They do not apply directly to the solutions of (\ref{Abreu4}) when $1<q<2$ because $w\not \in C^2(\Omega)$.
\end{rem}
Finally, we prove a universal positive lower bound for $\det D^2 u$.
\begin{lem}
\label{lowdet_lem}
 There is a universal constant $C_3>0$ such that
 $$\det D^2 u\geq C_3\quad\text{in } \Omega.$$
 \end{lem}
 \begin{proof}
 From the universal upper bound for $\det D^2 u$ in Lemma \ref{updet_lem},  we can construct an explicit barrier using the uniform convexity of $\Omega$ 
  to  show that, for a universal constant still denoted by $C_2$, 
  \begin{equation}
  \label{Dubound}
  |Du|\leq C_2 ~\text{in }\Omega.
  \end{equation}
  
  We use the Legendre transform in Proposition \ref{new2}. Recall that $$w^{\ast}=-\log \det D^2 u^{\ast}-1.$$
 When $F(x, z, p)=\frac{(|p|^2+\delta)^{\frac{q}{2}}}{q}+F^0(x, z)$, the term
\begin{eqnarray*}
F_{p_ip_j} u^{\ast, ij}=  u^{\ast, ij}\left[\frac{(|y|^2+\delta)^{\frac{q}{2}}}{q}\right]_{y_i y_j}
\end{eqnarray*}
can be absorbed into the left hand side of  \eqref{new-eq-leg} and we get, after a sign change of both sides
\beq\label{111}
u^{\ast, ij}\left[\log \det D^2 u^{\ast}+ \frac{(|y|^2+\delta)^{\frac{q}{2}}}{q}\right]_{y_i y_j}=F^0_z.
\eeq
Note that the right hand side of  \eqref{111} depends on $\sup_\Omega |u^{\ast}|$ and $\sup_\Omega |Du^{\ast}|$. 
Observe that 
$$u^{\ast, ij}\left[\log \det D^2 u^{\ast}+ \frac{(|y|^2+\delta)^{\frac{q}{2}}}{q} + Au^{\ast}\right]_{y_i y_j} = F^0_z + nA\geq 0$$
if $A$ is sufficiently large depending on $\sup_\Omega |u^{\ast}|$ and $\sup_\Omega |Du^{\ast}|$. We use the maximum principle to conclude that $\log \det D^2 u^{\ast}+ \frac{|y|^q}{q} + Au^{\ast}$ attains its maximum on $\p\Omega^\ast$. 
 If $y=Du(x)\in \p\Omega^*$, then 
$$\log \det D^2 u^\ast(y)= \log [\det D^2 u(x)]^{-1} =\log \psi (Du^*(y))\leq C.$$
It follows that $\log \det D^2 u^{\ast}\leq C$ from which we find $\det D^2 u\geq C_3= e^{-C}$. 
 \end{proof}
 \subsection{Proof of Theorem \ref{SBVq12e}}
 Before giving the proof of Theorem \ref{SBVq12e}, we state a main tool regarding H\"older estimates for the linearized Monge-Amp\`ere equation.
 By combining the global H\"older estimates for the linearized Monge-Amp\`ere equation in \cite[Theorem 1.4]{Le1} (in all dimensions, with right hand side being in $L^n$) and \cite[Theorem 1.2]{Le4} (in two dimensions, with right hand side being the divergence of a bounded vector field), we obtain the following theorem.
\begin{thm} [Global H\"older estimates for the linearized Monge-Amp\`ere equation]
\label{global-H}
Let $\Omega$ be a bounded, uniformly convex domain in $\R^{n}$ ($n=2$) with $\p\Omega\in C^{3}$. Let $\phi: \overline{\Omega}\rightarrow \R$, $\phi\in C^{0,1}(\overline{\Omega})\cap C^{2}(\Omega)$  be a convex function satisfying
$$0<\lambda\leq \det D^{2}\phi\leq \Lambda<\infty,~
\text{and}~\phi\mid_{\p\Omega}\in C^{3}.$$
Denote by $(\Phi^{ij})= (\det D^2\phi) (D^2\phi)^{-1}$ the cofactor matrix of $D^2 \phi$.
 Let $v\in C(\overline{\Omega})\cap C^2(\Omega)$ be the solution to the linearized Monge-Amp\`ere equation
 \begin{equation*}
 \left\{
 \begin{alignedat}{2}
   \Phi^{ij}D_{ij}v ~&=g+ \div G  ~&&\text{in} ~\Omega, \\\ 
 v&= \varphi ~&&\text{on}~ \p\Omega,
 \end{alignedat}
 \right.
\end{equation*}
 where $\varphi\in C^{\alpha}(\p\Omega)$ for some $\alpha\in (0, 1)$, $G\in L^{\infty}(\Omega, \R^n)$ and $g\in L^n(\Omega)$. Then $v\in C^{\alpha_1}(\overline{\Omega})$ 
with the estimate
$$\|v\|_{C^{\alpha_1}(\overline{\Omega})}\leq C\left(\|\varphi\|_{C^{\alpha}(\p\Omega)} + \|g\|_{L^{n}(\Omega)} + \|G\|_{L^{\infty}(\Omega)}\right)$$
where $\alpha_1$ and $C$ depend only on $\lambda$, $\Lambda$, $n$, $\alpha$, $\|\phi\|_{C^3(\p\Omega)}$, $\|\p\Omega\|_{C^3}$ and the uniform convexity of $\Omega.$ 
 $C$ also depends on $diam (\Omega)$.
  \end{thm}

\begin{proof}[Proof of Theorem \ref{SBVq12e}]
By Lemmas \ref{updet_lem} and \ref{lowdet_lem}, we have
$$0<C_3 \leq \det D^2 u_\e\leq C_2\quad \text{and } |u_\e| + |Du_\e| \leq C_2,\quad \text{in }\Omega.$$
We apply Theorem \ref{global-H} to the solution $w_\e$ of
 \begin{equation}
 \label{weeq}
 \left\{
 \begin{alignedat}{2}
   U_\e^{ij}D_{ij}w_\e~&=-div ((|Du_\e|^2 +\delta(\e))^{\frac{q-2}{2}} Du_\e) + F^0_z(x, u_\e)  ~&&\text{in} ~\Omega, \\\
 w_\e&= \psi ~&&\text{on}~ \p\Omega,
 \end{alignedat}
 \right.
\end{equation}
and find that  $w_\e\in C^{\alpha}(\overline{\Omega})$ with universal estimates and a universal $\alpha\in (0, 1)$. Now we apply the global $C^{2,\alpha}$ regularity for the Monge-Amp\`ere equation  (see  \cite{TW3})
 \begin{equation*}
 \left\{
 \begin{alignedat}{2}
   \det D^2 u_\e~&=w_\e^{-1}~&&\text{in} ~\Omega, \\\ 
 u_\e&= \varphi~&&\text{on}~ \p\Omega,
 \end{alignedat}
 \right.
\end{equation*}
to obtain 
$u_\e\in C^{2,\alpha}(\overline{\Omega})$ with universal estimates. As a consequence, the second order operator $U_\e^{ij} D_{ij}$ is uniformly elliptic with H\"older continuous coefficients with a universal exponent $\alpha$. 

\vskip 8pt
(a) The right hand side of (\ref{weeq})
is
$$f_\e:= -(|Du_\e|^2 +\delta(\e))^{\frac{q-2}{2}} \Delta u_\e - (q-2)(|Du_\e|^2 +\delta(\e))^{\frac{q-4}{2}} D_i u_\e D_j u_\e D_{ij} u_\e+ F^0_z(x, u_\e)\in C^{\alpha}(\overline{\Omega})  $$
with $\|f_\e\|_{C^{\alpha}(\overline{\Omega})}$ depending also on $\e$. Therefore, we can estimate the $C^{2, \alpha}(\overline{\Omega})$ norm of $w_\e$ from (\ref{weeq}). This gives
$u_\e \in C^{4, \alpha}(\overline{\Omega})$ with estimates depending on $\e$. It follows that we can estimate the $C^{1}(\overline{\Omega})$ norm of $f_\e$ depending also on $\e$. Using (\ref{weeq}) again, we find that $w_\e\in C^{2,\gamma}(\overline{\Omega})$ for all $\gamma<1$. Hence $u_\e \in C^{4, \gamma}(\overline{\Omega})$ with estimates depending on $\e$.

\vskip 8pt
(b) Assume $q\geq 2$. Then $f_\e\in C^{\gamma}(\overline{\Omega})$ for some universal $\gamma\in (0, 1)$ and
$\|f_\e\|_{C^{\gamma}(\overline{\Omega})}\leq C_4$
for some universal constant $C_4$. Therefore, we can estimate the $C^{2, \beta}(\overline{\Omega})$ norm of $w_\e$ from (\ref{weeq}) for $\beta:=\min\{\alpha,\gamma\}$. This gives
$u_\e \in C^{4, \beta}(\overline{\Omega})$ with a universal estimate
$$\|u_\e\|_{C^{4,\beta}(\overline{\Omega})}\leq C_5.$$

\vskip 8pt
 (c) Assume $1<q<2$. Note that $(|Du_\e|^2 +\delta(\e))^{\frac{q-2}{2}} Du_\e$ is H\"older continuous with a universal exponent $\gamma$ depending on $\alpha$ and $q>1$. Let $\beta=\min\{\alpha,\gamma\}$.
   Using (\ref{weeq}), we see that the universal $C^{1,\beta}(\overline{\Omega})$ estimates for $w_\e$ follows from \cite[Theorem 8.33]{GT}. Hence, we have the universal
   $C^{3,\beta}(\overline{\Omega})$ estimates for $u_\e$. 
\end{proof}


\section{Proof of Theorem \ref{RCthm}}
\label{RCsect}
\begin{proof}[Proof of Theorem \ref{RCthm}]
The proof of this theorem follows the strategy of that of \cite[Theorem 1.4]{Le7}.\\

(i) We solve
 (\ref{Abreuq})-(\ref{feq}) in $W^{4, s}(\Omega)$ $(s>n=2)$ by proving a priori $W^{4, s}(\Omega)$ estimates for uniformly convex $W^{4, s}(\Omega)$ solutions and then using the degree theory. Once the a priori $W^{4, s}(\Omega)$ estimates
 have been established (see, \ref{uew4s}), we can use a Leray-Schauder degree argument as in \cite[Theorem 2.1]{LeCPAM} to show the existence of a uniformly convex solution  $u_{\e}\in W^{4, s}(\Omega)$ (for all $s<\infty$) to the system (\ref{Abreuq})-(\ref{feq}).
 Thus, it suffices to prove these a priori estimates. 
 
 Let $u_\e\in W^{4, s}(\Omega)$ be a uniformly convex solution to (\ref{Abreuq})-(\ref{feq}). We will prove that, if $\e$ is sufficiently small, then
 \begin{equation}
 \label{uew4s}
 \|u_\e\|_{W^{4, s}(\Omega)} \leq C(\e, \varphi,\psi,\Omega, \Omega_0, s, \eta, \gamma).
 \end{equation}
In what follows, we call constants depending on $ \varphi,\psi,\Omega, \Omega_0, s, \eta, \gamma$ {\it universal}. Constants depending on $\e$ will be mentioned explicitly.

The key step in establishing (\ref{uew4s}) and in proving (ii) is to prove the universal bound, independent of $\e$, for $u_\e$ when $\e$ is sufficiently small. 

\vskip 10pt

{\it Step 1: Universal $L^{\infty}$ bound for $u_\e$.} We will prove that, if $\e$ is sufficiently small, then
\begin{equation}
\label{uebd}
 \|u_\e\|_{L^{\infty}(\Omega)} \leq C(\varphi,\psi,\Omega, \Omega_0, s, \eta, \gamma).
\end{equation}
Let 
\begin{equation}
\label{F1e}
F^1_\e(x, p) =\left[(|p|^2 +\delta(\e))^{\frac{q}{2}}/q-x\cdot p\right]\gamma(x).
\end{equation}
 Then, 
$$F^{1}_{\e, p_i p_j}=(|p|^2+\e)^{\frac{q-4}{2}} [(|p|^2+\delta(\e))\delta_{ij} + (q-2)p_i p_j]\gamma(x), $$
and, recalling that $\delta(\e)=\e$, 
\begin{equation}
\label{D2F1}
\min\{1, q-1\}(|p|^2+\e)^{\frac{q-2}{2}}\gamma(x) I_2 \leq (F^{1}_{\e, p_i p_j})\leq \max\{1, q-1\}(|p|^2+\e)^{\frac{q-2}{2}}\gamma(x) I_2,
\end{equation}
where $I_2$ is the identity $2\times 2$ matrix.
Therefore, the lowest eigenvalue of $ (F^{1}_{\e, p_i p_j})$
blows up when $p$ is small if $1<q<2$ or  $p$ large if $q>2$. This is a crucial difference between $F^1_\e$ and the function $F^1$ in \cite[Theorem 1.4]{Le7} where $F^1$ was assumed to satisfy $F^{1}_{p_i p_j}\leq C_\ast I_2$ for some universal constant $C_\ast$. Thus, we need to refine the analysis in \cite{Le7} to overcome the unboundedness of the Hessian of $F^1_\e$ in the $p$ variable.

Since $u_\e \leq \sup_{\p\Omega} u_\e=\sup_{\p\Omega} \varphi$ by convexity, to prove (\ref{uebd}), it suffices to prove that 
\begin{equation}
\label{ueom}
u_\e\geq -C \quad\text{in }\Omega.
\end{equation}
  Let us denote as in (3.1) in \cite{Le7}:
 $$ \tilde u =\varphi + \e^{\frac{1}{3n^2}} (e^{\rho}-1).$$
 Then
 \begin{eqnarray*}f_\e&=& \left[F^0_z(x, u_\e(x))- \frac{\p}{\p x_i} \left(\frac{\p F^1_\e}{\p  p_i}(x, Du_\e(x))\right)\right]\chi_{\Omega_0}(x) + \frac{1}{\e} (u_\e(x) -\tilde u(x))\chi_{\Omega\setminus \Omega_0}(x)\\
 &=& \left[ F_z^0(x, u_\e(x)) -F^1_{\e, p_i x_i} (x, Du_\e(x)) - F^1_{\e, p_i p_j}(x, Du_\e(x)) D_{ij}u_\e\right]\chi_{\Omega_0}(x) \\&&+ \frac{1}{\e} (u_\e(x) -\tilde u(x))\chi_{\Omega\setminus \Omega_0}(x).
 \end{eqnarray*}
As in (3.6) in \cite{Le7}, we have
\begin{equation}
\label{36eq}
\int_{\partial \Omega} \e((u_\e)_{\nu}^{+})^n dS \leq C+  \int_{\Omega} -f_\e (u_\e-\tilde u)dx.
\end{equation}
Let
$$M:= \sup_{\p\Omega}|\varphi|, \quad m:= \inf_{\Omega_0}\varphi + \inf_{\Omega_0}(e^\rho-1),\quad \bar\alpha:= \frac{\dist(\Omega_0,\p\Omega)}{\text{diam}(\Omega)}>0.$$

\vskip 5pt

{\bf Case 1:} If $$u_\e\geq \frac{m-M-1}{\bar\alpha}\quad \text{in } \Omega_0$$ then, by convexity, we obtain (\ref{ueom}).

\vskip 8pt

{\bf Case 2: } There is $z\in\Omega_0$ such that 
$$u_\e(z)\leq \frac{m-M-1}{\bar\alpha}<0.$$

We will show that \begin{equation}
\label{ueptil}
u_\e\leq \tilde u \quad \text{in } \Omega_0.
\end{equation}
Indeed, for any $x\in\Omega_0\setminus\{z\}$, let $y$ be the intersection of the ray $zx$ with $\p\Omega$. Then
$$x= \alpha z+ (1-\alpha) y\quad\text{where } \alpha=\frac{|x-y|}{z-y}\geq \bar\alpha.$$
Thus, by convexity
$$u_\e(x)\leq \alpha u_\e(z) + (1-\alpha) u_\e(y) \leq \bar\alpha u_\e(z) + M\leq m-1< \inf_{\Omega_0}\tilde u\leq \tilde u(x).$$
Therefore, we have (\ref{ueptil}). With (\ref{ueptil}), 
 we have
 \begin{equation}
  \label{39eq}
  \int_{\Omega_0} F^{1}_{\e, p_i p_j} D_{ij}u_\e(u_\e-\tilde u) dx \leq 0.
 \end{equation}
Let us now continue with the proof of (\ref{ueom}) in {\bf Case 2}. Using the convexity of $F^0$ (see (\ref{F0q})) and the universal boundedness of $\tilde u$, we get
\begin{equation}
\label{F0e}
\int_{\Omega_0}- F_z^0(x, u_\e(x)) (u_\e-\tilde u) dx\leq \int_{\Omega_0}- F_z^0(x, \tilde u(x)) (u_\e-\tilde u) dx\leq  C +  C_1\|u_\e\|_{L^{\infty}(\Omega_0)}.
\end{equation}
Since $u_\e$ is convex with $u_\e=\varphi$ on $\Omega$, we have the following gradient estimate (see, (3.1) in \cite{LeCPAM})
 \begin{equation}
 \label{uegrad}
 |Du_\e(x)|\leq \frac{\max_{\p\Omega} \varphi - u_\e(x)}{\dist(x,\p\Omega)}~\quad\text{ for } x\in\Omega.
 \end{equation}
 
 For $F^1_\e$ defined by (\ref{F1e}), 
$$F^{1}_{\e, p_i x_i}(x,p)=(|p|^2+\e)^{\frac{q-2}{2}} p_i\gamma_{x_i} -x_i\gamma_{x_i}-\gamma(x). $$
We recall that $\gamma$ is a nonnegative, Lipschitz function and it is a constant if $q>2$.
Hence, we have
 $$ |F^1_{\e, p_i x_i}(x, p)| \leq  C_2(|p| +1) \text{ for all }x\in\Omega_0~\text{and for each } i.$$
 This the only place where we need to assume $\gamma$ is constant when $q >2$.
 Thus, using (\ref{uegrad}), we can estimate in $\Omega_0$:
 \begin{eqnarray}
 \label{F1xp}
 |F^1_{\e, p_i x_i} (x, Du_\e(x)) (u_\e(x)-\tilde u(x) )| &\leq& C_2( |Du_\e(x)|+ 1)(|u_\e(x)| + C) \nonumber\\[4pt]
&\leq& C_3(|u_\e(x)|^2 + 1).
\end{eqnarray}
By Corollary 2.2 in \cite{Le7}, we have
\begin{equation}
\label{corl7}
\|u_\e\|_{L^{\infty}(\Omega_0)}\leq C +  C\int_{\Omega\setminus\Omega_0}|u_\e| dx.
\end{equation}
This
 together with $\|\tilde u\|_{L^{\infty}(\Omega)}\leq C$ gives
\begin{equation}
\label{ueom0}
\|u_\e\|_{L^{\infty}(\Omega_0)}\leq C +   C\int_{\Omega\setminus\Omega_0}|u_\e-\tilde u|^2 dx. 
\end{equation}
From (\ref{F0e}), (\ref{39eq}), (\ref{F1xp}) and (\ref{ueom0}), we find that
\begin{eqnarray}
\label{bounde1}
\int_{\Omega_0}- f_{\e} (u_\e-\tilde u) dx 
&=& \int_{\Omega_0} \left [F^1_{\e, p_i x_i} (x, Du_\e(x)) + F^1_{\e, p_i p_j}(x, Du_\e(x)) D_{ij}u_\e\right] (u_\e-\tilde u) \,dx \nonumber \\&&+ \int_{\Omega_0}- F_z^0(x, u_\e(x)) (u_\e-\tilde u) \,dx \nonumber\\
&\leq & C +  C_1\|u_\e\|_{L^{\infty}(\Omega_0)} + C_3\int_{\Omega_0}  (|u_\e|^2 + 1) \,dx \nonumber\\
&\leq& C\|u\|^2_{L^{\infty}(\Omega_0)} +C\nonumber\\
&\leq& C + C_4 \int_{\Omega\setminus\Omega_0}|u_\e-\tilde u|^2 \,dx.
\end{eqnarray}
Note that $f_\e= \frac{1}{\e} (u_\e-\tilde u)$ in $\Omega\backslash\Omega_0$. Hence, it follows from (\ref{36eq}) and (\ref{ueom0}) that
\begin{eqnarray*} 
\int_{\partial \Omega} \e  ((u_\e)^{+}_{\nu})^n dS +  \int_{\Omega\setminus\Omega_0} \frac{1}{2\e}|u_\e-\tilde u|^2 dx& \le& C  +  \int_{\Omega} -f_{\e} (u_\e-\tilde u ) dx +   \int_{\Omega\setminus\Omega_0} \frac{1}{2\e}|u_\e-\tilde u|^2 dx\nonumber \\&=& C + \int_{\Omega_0}- f_{\e} (u_\e-\tilde u ) dx+  \int_{\Omega\setminus\Omega_0} -\frac{1}{2\e}|u_\e-\tilde u|^2 dx \\
 &\leq& C + C_4 \int_{\Omega\setminus\Omega_0}|u_\e-\tilde u|^2 dx   -  \int_{\Omega\setminus\Omega_0} \frac{1}{2\e}|u_\e-\tilde u|^2 dx.
\end{eqnarray*} 
Therefore, if $\e\leq \e_0$ where $\e_0<1$ is universally small, then we get
\begin{equation}
\label{ep_1}
\int_{\partial \Omega} \e  ((u_\e)^{+}_{\nu})^n dS +  \int_{\Omega\setminus\Omega_0} \frac{1}{2\e}|u_\e-\tilde u|^2 dx \leq C
\end{equation}
Using (\ref{corl7}) and the universal boundedness of $\tilde u$, we obtain  (\ref{uebd}) as asserted.
\vskip 10pt
{\it Step 2: $W^{4, s}$ estimate for $u_\e$} if $\e\leq \e_0$.\\
From (\ref{uebd}) and gradient estimate (\ref{uegrad}), we find
\begin{equation}
\label{Duebd}
\|Du_\e\|_{L^{\infty}(\Omega_0)}\leq C_5.
\end{equation}
Therefore, recalling (\ref{D2F1}), we obtain the following refined estimate on the Hessian of $F^1_\e$ when evaluated at $Du_\e(x)$ where $x\in\Omega_0$:
\begin{equation}
\label{D2F1refined}
(F^{1}_{\e, p_i p_j}(x, Du_\e(x))) \leq C_\ast(\e) I_2~\quad\text{for all } x\in\Omega_0.
\end{equation}
where
\begin{equation*}  
C_\ast(\e) =\left\{
 \begin{alignedat}{1}
  (q-1) (C_5^2 +\e)^{\frac{q-2}{2}}  \|\gamma\|_{L^{\infty}(\Omega_0)}~&\text{if} ~  q\geq 2, \\\
\e^{\frac{q-2}{2}}\|\gamma\|_{L^{\infty}(\Omega_0)}~&\text{if}~1<q<2.
 \end{alignedat} 
  \right.
\end{equation*} 
From the universal a priori $L^{\infty}(\Omega)$ estimates (\ref{uebd}) 
for $u_\e$ and the refined Hessian estimate (\ref{D2F1refined}),
we can establish the a priori $W^{4, s}(\Omega)$ estimates (\ref{uew4s}) for $u_\e$ as in \cite[Theorem 4.1]{LeCPAM}. 
The proof of (i) is now complete.

\vskip 10pt

(ii) Let $u_\e\in W^{4, s}(\Omega)$ $(s>n)$ be a solution to (\ref{Abreuq})-(\ref{feq}) where $\e\leq \e_0$. Then (\ref{uebd}) and (\ref{Duebd}) hold.
Thus, for $x\in\Omega_0$, we have
$$|D_p F^1_{\e}(x, Du_\e)|= |Du_\e|  (|Du_\e|^2 +\e)^{\frac{q-2}{2}}\gamma(x) \leq (C_5^2 +\e)^{\frac{q-1}{2}}  
\|\gamma\|_{L^{\infty}(\Omega_0)}\leq  (C_5^2 +1)^{\frac{q-1}{2}}  \|\gamma\|_{L^{\infty}(\Omega_0)}. $$
Now, we argue exactly as in the proof of Theorem 1.4 (ii) in \cite{Le7}, which uses the universal boundedness of $|D_p F^1_{\e}(x, Du_\e)|$,  to complete the proof of (ii). 
We omit the details.
\end{proof}

\vglue 0.5cm
{\bf Acknowledgements.} The authors would like to thank the referee for helpful comments on the manuscript.

\end{document}